\documentclass[a4paper]{article}
\usepackage{fullpage}
\usepackage{amsmath}
\usepackage{amssymb}
\usepackage{amsthm}
\usepackage{graphicx}
\usepackage{bbm}
\usepackage{enumerate}
\usepackage{microtype}
\usepackage{tocbibind}
\usepackage{xcolor}
\usepackage{tikz}
\usepackage[none]{hyphenat}

\definecolor{mycol}{RGB}{120, 20, 167} 
\usepackage[colorlinks=true,urlcolor=blue,linkcolor=blue,citecolor=mycol,plainpages=false,pdfpagelabels]{hyperref}
\allowdisplaybreaks

\newtheorem{theorem}{Theorem}[section]
\newtheorem{proposition}[theorem]{Proposition}
\newtheorem{lemma}[theorem]{Lemma}
\newtheorem{corollary}[theorem]{Corollary}
\newtheorem{remark}[theorem]{Remark}
\newtheorem{notation}[theorem]{Notation}
\newtheorem{definition}[theorem]{Definition}

\newcommand{\mex}{\operatorname{mex}}

\newcommand{\otype}{\operatorname{otyp}}
\newcommand{\ordinali}{\operatorname{On_2}}
\newcommand{\ordinal}{\operatorname{On}}

\newcommand{\boxedexp}[1]{^{%
		\tikz[baseline=(char.base)]{
			\node[
			draw, 
			rounded corners=2pt, 
			inner sep=1.5pt, 
			font=\normalfont 
			] (char) {$\scriptstyle #1$}; 
		}%
}}

\title{Some quadratically closed fields of nimbers}
\author{Lucia Rîșnoveanu \\[2mm]
\footnotesize Faculty of Mathematics and Computer Science, University of Bucharest,\\
\footnotesize Academiei 14, 010014 Bucharest, Romania \\[2mm]
\footnotesize Email: lucia.risnoveanu@yahoo.ro\\
}
\date{}

\begin{document}
	
	\maketitle
	\begin{abstract}
		In 1976, J. H. Conway introduced Nim arithmetic which establishes an algebraically closed field structure over the class of ordinals and proved that the first transcendental ordinal is $\omega^{\omega^\omega}$. The problem of finding the next transcendental ordinal is still open. Two years later, H. Lenstra proved that $\varepsilon_0$ is the next quadratically closed field ordinal. In this paper, we prove that $\{\varepsilon_\alpha \mid \alpha \leq \omega^{\omega^\omega} \}$ are the next quadratically closed field ordinals.
		
		

	\end{abstract}

	\section{Introduction}
		
		Unlike standard ordinal arithmetic defined by Cantor, which lacks the properties of a commutative group or a field, Nim addition and multiplication—introduced by J. H. Conway in the book  \textit{On Numbers and Games}  \cite{onag}—endow ordinals with a surprisingly rich algebraic structure.
	
		Nim arithmetic forms an algebraically closed field structure on the class of ordinals. J. H. Conway proved that $ \omega^{\omega^{\omega}} $  is the algebraic closure of $2$. Thus, Nim arithmetic provides a constructive way to determine the algebraic closure of $ \mathbb{F}_2 $.

		J. H. Conway raised the question regarding the value of the next algebraically closed field ordinal, a problem that remains open to this day. If this next algebraically closed field were to be determined, then Nim arithmetic would provide a constructive model for the algebraic closure of $ \mathbb{F}_2(X) $.

		We will use results about standard ordinal arithmetic and Galois theory. We will also use results about Nim arithmetic, developed by J. H. Conway in \cite{onag} and explained in detail in \cite{siegel}. These results are written in Section~\ref{sec: prel}.
		
		In Subsection~\ref{subsec: 3.1} we prove that the algebraic closure of a field ordinal and its $n$-closure, for $n$ a positive integer, are all ordinals, and we show the relation between them. Also, we prove that $\omega_1$ is an algebraically closed field ordinal and that there is an uncountable infinity of countable transcendental ordinals. 
		In the rest of Section~\ref{sec: 3} we explain Lenstra's result \cite{nim-mult} that $\varepsilon_0$ is the quadratic closure of $\omega^{\omega^\omega}(\omega^{\omega^\omega})$. We give detailed proofs for completeness, as we later make serious use of some (ideas) of these lemmas.
		
		A. Siegel \cite{siegel} notes that $\varepsilon_1$ is the quadratic closure of $\varepsilon_0(\varepsilon_0)$ -- a result originally left as an exercise by H. Lenstra in \cite{nim-mult}. To date, no further work has been published concerning the quadratic closure of  $\varepsilon_1(\varepsilon_1)$, leaving it an open question whether this pattern still holds. In Section~\ref{sec: 4} we prove that this pattern still holds at least until $\varepsilon_{\omega^{\omega^{\omega}}}$. We also prove the identity $\varepsilon_{\omega^{\omega^{\omega}}} \boxedexp{3} = \omega^{\omega^{\omega^{\omega}+1}}  $.

		Throughout this paper, we will use the notation for Nim-sum and Nim-product used in \cite{siegel}.

	\section{Preliminaries} \label{sec: prel}

	\subsection{Standard ordinal arithmetic}
	
	We assume the basic properties of the usual operations on ordinals. In particular, for all $ \alpha, \ \beta, \ \gamma $ ordinals, $ (\alpha^\beta)^\gamma = \alpha^{\beta \times \gamma}$.
	
	\begin{lemma} \label{1+v=v}
		Let $v$ be an ordinal such that $ v \geq \omega $. Then $1+v = v$. Moreover, $n+v = v$ for any $n \in \omega$.
	\end{lemma}
	
	\begin{proof}
		We prove by induction on $v$. If $v = \omega$, then $1+v =\sup \{ 1+n \mid n \in \omega \} =v$. 
		If $v = u+1$, then $1+(u+1) =(1+ u)+1 = u+1=v$.
		If $v$ is a limit ordinal, then $1+v = \lim \{ 1+u  \mid u< v\} = \lim \{ u \mid u<v \} =v$. The fact that $n+v = v$ for any $n \in \omega$ follows immediately by induction on $n$.
	\end{proof}
	
	\begin{lemma} \label{ineg-ord-pt-claim}
		Let $\alpha$ be a limit ordinal and an ordinal $\beta < 2^{\alpha}$. Then $\beta + 2^{\alpha}=2^{\alpha} $.
	\end{lemma}
	
	\begin{proof}
		We have that $ 2^{\alpha} \leq \beta+ 2^{\alpha} \leq 2^{\beta} + 2^{\alpha} = \sup \{ 2^{\beta} + 2^{\delta} \mid \delta < \alpha \} \leq \sup \{  2^{\delta+1} \mid \delta < \alpha \} = 2^{\alpha} $.

	\end{proof}
	
	\begin{definition}
		An ordinal $\alpha$ is called an $\varepsilon$-number if $\alpha = 2^{\alpha}$.
	\end{definition}
	
	\begin{remark} \label{obs-epsilon-nr}
		An $\varepsilon$-number is a limit ordinal.
	\end{remark}
	
	\begin{proof} 
		Let $\varepsilon$ be an $\varepsilon$-number. Then $\varepsilon \geq \omega $. We denote $\varepsilon = \omega \times x + y$, where $x \geq 1$, $ y<\omega $, $x = 1+z$, $z \geq 0$. Then $ \varepsilon = 2^{\varepsilon} = \omega^x \times 2^y = \omega \times (\omega^z \times 2^y) $, so $y = 0$. Thus $\varepsilon = \omega \times x$ for an ordinal $x \geq 1$, so $\varepsilon$ is a limit ordinal.
	\end{proof}

	\begin{lemma} \label{lema-epsilon}
		Let $\varepsilon$ be an $\varepsilon$-number and a nonzero ordinal $v < \varepsilon$. Then $v \times \varepsilon = \varepsilon$.
	\end{lemma}
	
	\begin{proof}
		From Remark~\ref{obs-epsilon-nr}, $\varepsilon$ is a limit ordinal. Then $v+\varepsilon \leq 2^v + 2^{\varepsilon} = \sup \{ 2^v+2^{\delta} \mid \delta < \varepsilon \} \leq \sup \{ 2^{\delta +1} \mid \delta < \varepsilon \} = 2^{\varepsilon} = \varepsilon $. Thus $v \times \varepsilon \leq 2^v \times 2^{\varepsilon} = 2^{v + \varepsilon} \leq 2^{\varepsilon} = \varepsilon $, from which it follows that $ v \times \varepsilon = \varepsilon$.  
	\end{proof}
	
	\begin{proposition} \label{prop-eps-nr-indice-leg}
		We have that $\{ \varepsilon \mid 2^{\varepsilon} = \varepsilon \} = \{ \omega \} \cup \{ \varepsilon \mid \omega ^ \varepsilon = \varepsilon \}  $.
	\end{proposition}
	
	\begin{proof}
		Obviously $  \{ \varepsilon \mid 2^{\varepsilon}= \varepsilon \} \supset \{ \omega \} \cup \{ \varepsilon \mid \omega ^ \varepsilon = \varepsilon \}    $. Let $\varepsilon > \omega   $ such that $ 2^{\varepsilon} = \varepsilon $. Then, from Lemma~\ref{lema-epsilon}, $\omega \times \varepsilon = \varepsilon$, so $ \omega ^ \varepsilon=  2^{\omega \times \varepsilon} = 2^\varepsilon = \varepsilon$.
	\end{proof}

	We enumerate the set $ \{ \varepsilon \mid \omega ^ \varepsilon = \varepsilon \}  $ as $ (\varepsilon_\alpha)_{\alpha \in \ordinal } $.

	\begin{proposition} \label{epsilon-0}
		Let $(a_n)_{n \in \omega}$ be the sequence with $a_0 = \omega$, $a_{n+1} = \omega^{a_n}$. Then $\varepsilon_0 = \sup \{ a_n \mid n \in \omega \}$.
	\end{proposition}
	
	\begin{proof} 
		Obviously $(a_n)$ is an increasing sequence and $\omega^{\sup \{ a_n \mid n \in \omega \}} = \sup \{ \omega^{a_n} \mid n \in \omega \} = \sup \{ a_{n+1} \mid n \in \omega \} = \sup \{ a_n \mid n \in \omega \} $. Thus $ \varepsilon_0 \leq \sup \{ a_n \mid n \in \omega \} $. 
		
		We prove by induction on $n$ that $ \varepsilon_0 \geq a_n $. Since $\varepsilon_0 \geq 1$, it follows that $\varepsilon_0 = \omega^{\varepsilon_0} \geq  \omega$. If $ \varepsilon_0 \geq a_n $, then $\varepsilon_0 = \omega^{\varepsilon_0} \geq  \omega^{a_n} = a_{n+1}$. 
		
		Thus $\varepsilon_0 \geq \sup \{ a_n \mid n \in \omega \}$.
	\end{proof}
	
	\begin{proposition} \label{epsilon-0-tau}
		Let $\tau := \omega^{\omega^{\omega}}$. Let $(a_n)_{n \in \omega}$ be the sequence with $a_0 = \omega$, $a_{n+1} = \omega^{a_n}$. Let $(b_n)_{n \in \omega}$ be the sequence with $b_0 = \tau$, $b_{n+1} = \tau^{b_n}$. Then $\sup \{ a_n \mid n \in \omega \}  = \sup \{ b_n \mid n \in \omega \} $.
	\end{proposition}

	\begin{proof}
		Obviously $a_n \leq b_n$ for any $ n \in \omega $. Thus $\sup \{ a_n \mid n \in \omega \}  \leq \sup \{ b_n \mid n \in \omega \} $.
		
		We prove by induction that $b_n \leq a_{n+2}$ for any $n \in \omega$. We have $ b_0=a_2 $. If $b_n \leq a_{n+2}$, then $ b_{n+1} = \omega^{\omega^{\omega} \times b_n} \leq \omega^{\omega^{\omega} \times a_{n+2}} = \omega^{\omega^{\omega} \times \omega^{a_{n+1}}} = \omega^{\omega^{\omega+a_{n+1}} }$. But, from Lemma~\ref{1+v=v}, $ \omega + a_{n+1} = \omega + \omega^{a_n} = \omega + \omega^{1+a_n} = \omega \times (1+\omega^{a_n}) = \omega \times \omega^{a_n} = \omega ^{1+a_n}= \omega^{a_n} = a_{n+1}$, so $b_{n+1} \leq \omega^{\omega^{a_{n+1}} } = a_{n+3}$.
		
		Obviously $(a_n)_n$ and $ (b_n)_n $ are increasing, so $\sup \{ a_n \mid n \in \omega \}  \geq \sup \{ b_n \mid n \in \omega \} $.
	\end{proof}
	
	\subsection{Galois theory}
	
	\begin{lemma} \label{corp-num}
		Let $K$ be a countable field. Then $\overline{K}$ is countable.
	\end{lemma}
	
	\begin{proof}
		See \cite{allufi}.
	\end{proof}
	
	\begin{definition}
		Let $n \in \mathbb{N}^*$. We say that a field $F$ is $n$-closed if every non-zero polynomial in $F[X]$ of degree at most $n$ has at least one root in $F$.
	\end{definition}
	
	\begin{definition} 
		For a field $K$ and a positive integer $n$, we denote $\overline{K}^n = \bigcap_{F \text{ n-closed field} ,K \subset F \subset \overline{K}} F$.
		We call $\overline{K}^n$ the $n$-closure of $K$.
	\end{definition}
	
	\begin{proposition} \label{extindere-alg-corp-perfect}
	Let $K$ be a perfect field and $K \hookrightarrow L$ an algebraic extension. Then $L$ is a perfect field.
	\end{proposition}

	\begin{proof}
		See \cite{allufi}.
	\end{proof}
	
	\begin{definition}
		Let $K \hookrightarrow L$ be an algebraic field extension. An element $x \in L$ is purely inseparable over $K$ if there exists $d \in \mathbb{N}^*$ such that $(X-x)^d$ is the minimal polynomial of $x$ over $K$. The extension $K \hookrightarrow L$ is purely inseparable if every element of $L$ is purely inseparable over $K$.  
	\end{definition}

	\begin{proposition} \label{pur-insep}
		Let $K$ be a field  of characteristic 2 and $K \hookrightarrow L$ an algebraic and purely inseparable field extension. Let $x \in L $. Then there exists $n \in \mathbb{N} $ such that the minimal polynomial of $x$ over $K$ is $(X-x)^{2^n}$. 
	\end{proposition}
	
	\begin{proof}
		Since $x$ is purely inseparable over $K$, it follows that there exists $d \in \mathbb{N} $ such that the minimal polynomial of $x$ over $K$ is $(X-x)^d$. Let $d = 2^r s$, where $d,s \in \mathbb{N}$ and $s$ is odd. Then $(X^{2^r}-x^{2^r})^s \in K[X]$, so $s x^{2^r} \in K$, from which it follows that $x^{2^{r}} \in K$. Since $(X^{2^r}-x^{2^r})^s \in K[X]$ is irreducible, it follows that $s=1$, so $d=2^r$. 
	\end{proof}
	
	\begin{proposition} \label{corolar-pur-insep}
		For a field $K$ of characteristic 2, we denote by $L$ a perfect closure of it. Then $K \hookrightarrow L$ is a purely inseparable field extension.
	\end{proposition} 
	
	\begin{proof}
		For every $n \in \mathbb{N}$, let $K_n = \{x \in \overline{K} \mid x^{2^n} \in K  \}$. Then $K_n$ is a subfield of $\overline{K}$ and $K_n \subset K_{n+1}$ for every $n \in \omega$. Let $F = \bigcup_{n \in \omega} K_n $. We have that $F$ is the smallest perfect subfield of $\overline{K}$ containing $K$, relative to the inclusion relation and so a perfect closure of $K$. Thus $L$ and $F$ are $K$-isomorphic. Since for every $x \in F $ there exists $n \in \omega $ such that $x^{2^n} \in K$, it follows that $K \hookrightarrow F$ is purely inseparable, hence $K \hookrightarrow L$ is a purely inseparable field extension.
	\end{proof}

	\subsection{Nim arithmetic}
	
	\begin{definition} 
		Let $S$ be a set of ordinals. The minimal excluded value of S is defined as:
		$$\mex(S) := \min \{ \alpha  \mid \alpha \notin S \}.$$
	\end{definition}

		In this article we will use the notation:
	$$ \{P_1(x), P_2(x), \dots , P_n(x) \mid x\} = \{P_1(x) \mid x \} \cup \{P_2(x) \mid x \} \cup \dots \cup \{P_n(x) \mid x \}.$$
	
	\begin{definition}
		Let $\alpha,\  \beta$ be two ordinals. Their Nim-sum and Nim-product are defined as: 
		$$\alpha \oplus \beta = \mex \{ \alpha' \oplus \beta, \alpha \oplus \beta' \mid \alpha' < \alpha, \beta' < \beta \};$$
		$$\alpha \otimes \beta = \mex \{ (\alpha' \otimes \beta) \oplus (\alpha \otimes \beta') \oplus (\alpha' \otimes \beta') \mid \alpha' < \alpha, \beta' < \beta \}.$$
	\end{definition}
	
	\begin{notation}
		For any ordinal $\alpha$ and positive integer $n$, $\alpha^{\boxedexp{n}} := \underbrace{\alpha \otimes \dots \otimes \alpha}_{ n \text{ times } \alpha}$.
	\end{notation}
	
	\begin{proposition}
		The class of ordinals with the Nim operations is a Field (i.e., a field whose domain is a proper Class) of characteristic 2, which is called $\ordinali$.
	\end{proposition}
	
	\begin{proof}
		See \cite[Chapter 6]{onag}.
	\end{proof}
	
	The following theorem describes ordinals that are not groups/rings/fields/algebraically closed fields as the ``simplest" extension over themselves and also states that in certain conditions, Nim arithmetic is the same as standard ordinal arithmetic. 
	
	\begin{theorem} \label{set}(The simplest extension theorem)
		Let $\gamma$ be a nonzero ordinal. 
		\begin{enumerate}[a)]
			\item If $(\gamma, \oplus)$ is not a group, then $\gamma = \alpha \oplus \beta$, where $(\alpha, \beta)$ is the
			lexicographically least pair of ordinals with $\alpha \oplus \beta \notin \gamma$.
			\item If $(\gamma, \oplus)$ is a group, then $(\gamma \times \beta) \oplus \alpha = (\gamma \times \beta)+\alpha$, for all ordinals $\beta$ and all $\alpha < \gamma$.
			\item If $(\gamma, \oplus)$ is a group, but $(\gamma,\oplus,\otimes)$ is not a ring, then $\gamma = \alpha \otimes \beta$, where $(\alpha, \beta)$ is the
			lexicographically least pair of ordinals with $\alpha \otimes \beta \notin \gamma$.
			\item If $(\gamma,\oplus,\otimes)$ is a ring, and let $\delta \leq \gamma$ be an ordinal such that $(\delta,\oplus)$ is a group and every nonzero $\alpha \in \delta$ has a Nim inverse in $\gamma$, then $\gamma \otimes \alpha = \gamma \times \alpha$ for all $\alpha \in \delta$.
			\item If $(\gamma,\oplus,\otimes)$ is a ring, but not a field, then $\gamma \otimes \alpha = 1 $, where $\alpha$ is the
			least nonzero ordinal with no Nim inverse in $\gamma$.
			\item If $(\gamma,\oplus,\otimes)$ is a field and $n \in \mathbb{N}$ such that every polynomial of degree at most $n$ has a Nim root in $\gamma$, then for all $ \alpha_0, \dots, \alpha_n < \gamma$, we have that
			$\bigoplus_{i=0}^n (\gamma^{\boxedexp{i}} \otimes \alpha_i) = \sum_{i=n}^0 (\gamma^i \times \alpha_i)$. 
			\item If $(\gamma,\oplus,\otimes)$ is a field but not algebraically closed, then $\gamma$ is a Nim root of the lexicographically least polynomial over $\gamma$ with no Nim root in $\gamma$. This polynomial is also the lexicographically least irreducible polynomial over $\gamma$.
			\item If $(\gamma,\oplus,\otimes)$ is an algebraically closed field and $n \in \mathbb{N}$, then for all $ \alpha_0, \dots, \alpha_n < \gamma$, we have that
			$\bigoplus_{i=0}^n (\gamma^{\boxedexp{i}} \otimes \alpha_i) = \sum_{i=n}^0 (\gamma^i \times \alpha_i)$. 
			\item If $(\gamma,\oplus,\otimes)$ is an algebraically closed field, then $\gamma$ is transcendental over $\gamma$. Ordinals that are algebraically closed field are called transcendental.
			
		\end{enumerate}
	\end{theorem}
	
	\begin{proof}
		See \cite[Theorem 4.3 and Lemma 4.4]{siegel}.
	\end{proof}
	
	Due to the next theorem, we can always say when an ordinal is a group.
	
	\begin{theorem} \label{thm-grup}
		\begin{enumerate}[a)]
			\item Ordinals that are groups are precisely the powers of 2.
			
			\item Each ordinal can be written uniquely as a finite sum of descending powers of 2,
			and it is the same sum in both senses ($\oplus$ and $+$).
		\end{enumerate}
	\end{theorem}

	\begin{proof}
		See \cite[Theorem 48]{onag}.
	\end{proof}

	\begin{lemma} \label{lema-produs-grupuri}
		Let $\alpha$ be a group ordinal. Let $\beta $ be a nonzero ordinal. Then $\beta $ is a group if and only if $ \alpha \times \beta$ is a group. 
	\end{lemma}
	
	\begin{proof}
		From Theorem~\ref{thm-grup} a) it follows that there is an ordinal $a$ such that $\alpha = 2^a$.
		
		If $\beta$ is a group, then from Theorem~\ref{thm-grup} a) it follows that there is an ordinal $b$ such that $\beta = 2^b$, so $ \alpha \times \beta = 2^{a+b}$ is a group.
		
		If $\alpha \times \beta$ is a group, then from Theorem~\ref{thm-grup} a) it follows that there is an ordinal $c$ such that $\alpha \times \beta = 2^c$. Then $c \geq a$, so there is an ordinal $b$ such that $c = a+b$, so $\beta = 2^b$ is a group.
	\end{proof}

	\begin{theorem} \label{primul-transcendent}
		The first transcendental ordinal is $\omega^{\omega^\omega}$. 
	\end{theorem}
	
	\begin{proof}
		See \cite[Theorem 4.6]{siegel}.
	\end{proof}
	
	\begin{notation}
		Let $\tau := \omega^{\omega^\omega}$.
	\end{notation}
	
	As opposed to group ordinals, field ordinals do not have a nice description and the next field ordinal after a given field ordinal depends on the nature of the field ordinal. 
	
	\begin{proposition} \label{urm-corp-1}
		If $\alpha$ is a transcendental ordinal, then the next field ordinal is $\alpha^{\omega + \omega \times \alpha}$. Moreover, $ \alpha^{\omega + \omega \times \alpha} = \alpha(\alpha) $. We have that, for all $\alpha'< \alpha$, 
		
		\begin{itemize}
			\item[\small$\bullet$] $\alpha^{\omega + \omega \times \alpha'} = \{ P(\alpha) \oslash (\bigotimes_{i=1}^m (\alpha \oplus \beta_i)^{\boxedexp{m_i}}) \mid m \geq 0,\ P \in \alpha[x],\ \beta_i < \alpha',\ m_i \in \mathbb{N}^*, \text{ for every } i \in \{1,\dots, m\} \}$ is a ring in which the smallest element without an inverse is $\alpha + \alpha'$;
			
			\item[\small$\bullet$] $\alpha^{\omega + \omega \times \alpha'} = 1 \oslash (\alpha \oplus \alpha') $;
			
			\item[\small$\bullet$] For all $n \in \omega $:
			
			\begin{itemize}
				\item[\small$\bullet$] $\alpha^{\omega + \omega \times \alpha' + n} = \{ P(1 \oslash (\alpha + \alpha')) \oplus \bigoplus_{i=1}^m P_i(1 \oslash (\alpha + \beta_i)) \oplus Q(\alpha) \mid Q \in \alpha[x],\ P \in \alpha[x] \text{ of degree at most } n,\  m \geq 0,\ P_i \in \alpha[x],\text{ for every } i \in \{1,\dots, m\},\ \beta_i < \alpha', \text{ for every } i \in \{1,\dots, m\} \}$;
				\item[\small$\bullet$] $\alpha^{\omega + \omega \times \alpha' + n} \otimes a = \alpha^{\omega + \omega \times \alpha' + n} \times a$, for all $ a < \alpha$; 
				\item[\small$\bullet$] $\alpha^{\omega + \omega \times \alpha' + n} = (\alpha^{\omega + \omega \times \alpha'})^{\boxedexp{n+1}}$.
			\end{itemize}

		\end{itemize}

	\end{proposition}

	\begin{proof}
		
		We prove by induction on $\alpha' < \alpha$ that:
		
		\begin{itemize}
			\item[\small$\bullet$] $\alpha^{\omega + \omega \times \alpha'}$ is a ring in which the smallest element without an inverse is $\alpha + \alpha'$;
			
			\item[\small$\bullet$] $\alpha^{\omega + \omega \times \alpha' + n} = \{ P(1 \oslash (\alpha + \alpha')) \oplus \bigoplus_{i=1}^m P_i(1 \oslash (\alpha + \beta_i)) \oplus Q(\alpha) \mid Q \in \alpha[x],\ P \in \alpha[x] \text{ of degree at most } n,\  m \geq 0,\ P_i \in \alpha[x],\text{ for every } i \in \{1,\dots, m\},\ \beta_i < \alpha', \text{ for every } i \in \{1,\dots, m\} \}, \text{ for every } n \in \mathbb{N}$.
		\end{itemize}

		For $\alpha' = 0$:
		
		From Theorem~\ref{set} h), we know that
		$a_0 \oplus (\alpha \otimes a_1) \oplus (\alpha^{\boxedexp{2}} \otimes a_2) \oplus \dots \oplus (\alpha^{\boxedexp{n}} \otimes a_n) = (\alpha^n \times a_n)+  \dots + (\alpha \times a_1) + a_0 $, for every $ n \geq 0,\  a_0, \dots, a_n < \alpha$.

		It follows that $\alpha^\omega$ is a ring, because:
		\begin{align*}
			\alpha^\omega &= \{ (\alpha^n \times a_n)+  \dots + (\alpha \times a_1) + a_0 \mid n \geq 0,\ a_0, \dots, a_n < \alpha \} \\ &
			= \{ a_0 \oplus (\alpha \otimes a_1) \oplus \dots \oplus (\alpha^{\boxedexp{n}} \otimes a_n) \mid n \geq 0,\ a_0, \dots, a_n < \alpha \} \\ &
			= \{ P(\alpha) \mid P \in \alpha[x]  \}.
		\end{align*}

		Since $\alpha$ is a field, any $ \gamma < \alpha$ has an inverse in $\alpha$, and thus also in $\alpha^\omega$. Assume towards a contradiction that there exists $ \beta \in \alpha^\omega$ such that $\beta \otimes \alpha = 1$. Let $\beta = a_0 \oplus (a_1 \otimes \alpha) \oplus \dots \oplus (a_n \otimes \alpha^{\boxedexp{n}})$, $n \geq 0,\ a_0, \dots, a_n < \alpha$.
		
		Then
		$1 \oplus (a_0 \otimes \alpha) \oplus (a_1 \otimes \alpha^{\boxedexp{2}}) \oplus \dots \oplus (a_n \otimes \alpha^{\boxedexp{n+1}}) = 0$, hence $\alpha$ is algebraic over $\alpha$, a contradiction.
		
		Thus $\alpha^\omega$ is a ring in which the smallest element without an inverse is $\alpha$. Then $\alpha^\omega = 1 \oslash \alpha$.
		
		We prove by induction on $n$ that:
		
		\begin{itemize}
			\item[\tiny$\bullet$] $\alpha^{\omega + n} = \{ P(1 \oslash \alpha) \oplus Q(\alpha) \mid Q \in \alpha[x],\ P \in \alpha[x] \text{ of degree at most } n \}$;
			
			\item[\tiny$\bullet$] $\alpha^{\omega + n} \otimes a = \alpha^{\omega + n} \times a $, for every $ a < \alpha$;
			
			\item[\tiny$\bullet$] $\alpha^{\omega + n} = (\alpha^\omega)^{\boxedexp{n+1}}$.
			
		\end{itemize}

		For $n = 0$, for the first and third bullet point, we have proved this. The second bullet point will not need a base case.
		
		Assume it is true for every $ m < n$. We prove it for $n$.
		
		We know that $\alpha^{\omega+n} = \{ \alpha^{\omega+n-1} \times a_{n-1} + \dots + \alpha^\omega \times a_0 + a \mid a_0, \dots, a_{n-1} < \alpha, a < \alpha^\omega \}$.

		Because $\alpha^\omega, \alpha^{\omega+1}, \dots, \alpha^{\omega+n-1}$ are groups and $a < \alpha^\omega, \alpha^\omega \times a_0 + a < \alpha^{\omega+1}, \dots$, we have from Theorem~\ref{set} b) that
		$\alpha^{\omega+n-1} \times a_{n-1} + \dots + \alpha^\omega \times a_0 + a = (\alpha^{\omega+n-1} \times a_{n-1}) \oplus \dots \oplus (\alpha^\omega \times a_0) \oplus a$.

		But, by induction, we have that
		$\alpha^{\omega+n-1} \times a_{n-1} + \dots + \alpha^\omega \times a_0 + a = (\alpha^{\omega+n-1} \otimes a_{n-1}) \oplus \dots \oplus (\alpha^\omega \otimes a_0) \oplus a$.
		
		Since by induction we know that $\alpha^{\omega+m-1} = (\alpha^\omega)^{\boxedexp{m}}$, for every $1 \leq m \leq n$ and $\alpha^\omega = 1 \oslash \alpha$, it follows that: 
		
		\begin{align*}
			\alpha^{\omega+n} &= \{ (\alpha^{\omega +n-1} \otimes a_{n-1}) \oplus \dots \oplus (\alpha^\omega \otimes a_0) \oplus a \mid a_0, \dots, a_{n-1} < \alpha, a < \alpha^\omega \} \\ &
			= \{ P(1 \oslash \alpha) \oplus Q(\alpha) \mid Q \in \alpha[x],\ P \in \alpha[x] \text{ of degree at most } n \}.
		\end{align*}

		Let $a<\alpha$. To prove that $\alpha^{\omega+n} \otimes a = \alpha^{\omega+n} \times a$,
		assume that it is true for any $a' < a$.

		We know that:
		
		\begin{align*}
			\alpha^{\omega+n} \otimes a &= \mex \{ (\beta \otimes a) \oplus (\alpha^{\omega+n} \otimes a') \oplus (\beta \otimes a') \mid \beta < \alpha^{\omega+n}, a' < a \} \\ &
			= \mex \{ (\alpha^{\omega+n} \times a') \oplus (\beta \otimes (a \oplus a')) \mid \beta < \alpha^{\omega+n}, a' < a \}.
		\end{align*}

		Let $a'<a$. Since $a \oplus a' \neq 0$ and $a \oplus a'$ has an inverse in $\alpha$.
		Put $ \zeta := a \oplus a'$. 
		
		Since $\alpha^{\omega+n} = \{ P(1 \oslash \alpha) \oplus Q(\alpha) \mid Q \in \alpha[x],\ P \in \alpha[x] \text{ of degree at most } n \}$ and $f: \alpha[x] \to \alpha[x]$ defined by $f(P)(x) = P(x) \otimes \zeta$ is bijective and degree-preserving, it follows that $\{ \beta \otimes \zeta \mid \beta < \alpha^{\omega+n} \} = \alpha^{\omega+n}$.
		
		Thus $\alpha^{\omega+n} \otimes a = \mex \{ (\alpha^{\omega+n} \times a') \oplus \beta \mid \beta < \alpha^{\omega+n}, a' < a \}$
		$= \mex \{ (\alpha^{\omega+n} \times a') + \beta \mid \beta < \alpha^{\omega+n}, a' < a \} = \alpha^{\omega+n} \times a$,
		because $\alpha^{\omega+n}$ is a group and from Theorem~\ref{set} b).
		
		Since $\alpha$ is a group, it follows from Theorem~\ref{thm-grup} that $\alpha^{\omega+n}$ is a group. We prove that $(\alpha^\omega, \alpha^{\omega+n-1})$ is the lexicographically least pair of elements from $\alpha^{\omega+n}$ with the Nim product greater than or equal to $\alpha^{\omega+n}$.
		From the characterizations of $\alpha^{\omega+n-1}, \ \alpha^{\omega+n}, \ \alpha^\omega$ and from the fact that $\alpha^\omega = 1 \oslash \alpha$, we have that for any $ \beta < \alpha^\omega, \gamma < \alpha^{\omega+n}: \beta \otimes \gamma < \alpha^{\omega+n}$ and for any $ \beta < \alpha^{\omega+n-1}: \alpha^\omega \otimes \beta < \alpha^{\omega+n}$.
		
		It remains to prove that $\alpha^\omega \otimes \alpha^{\omega+n-1} \geq \alpha^{\omega+n}$.
		By induction, $\alpha^{\omega+n-1} = (\alpha^\omega)^{\boxedexp{n}}$, so $\alpha^\omega \otimes \alpha^{\omega+n-1} = (\alpha^\omega)^{\boxedexp{n+1}} = (1 \oslash \alpha)^{\boxedexp{n+1}}$.
		If, towards a contradiction: $(1 \oslash \alpha)^{\boxedexp{n+1}} = P(1 \oslash \alpha) \oplus Q(\alpha)$, where $Q \in \alpha[x]$ and $P \in \alpha[x]$ of degree at most $n$, then, multiplying by $\alpha^{\boxedexp{n+1}}$, it follows that there exists a nonzero $ R \in \alpha[x]$ such that $R(\alpha) = 0$, a contradiction with the fact that $\alpha$ is transcendental.
		Then, $\alpha^{\omega+n} = \alpha^\omega \otimes \alpha^{\omega+n-1} = (\alpha^\omega)^{\boxedexp{n+1}}$ and the induction is finished and the verification step for the induction on $\alpha'$ is proved.
		
		Let us prove the induction step:

		If $\alpha'$ is a successor ordinal, let $\alpha' = \beta^+$.
		
		Then:
		\begin{align*}
			\alpha^{\omega + \omega \times \alpha'} &= \sup \{ \alpha^{\omega + \omega \times \beta + n} \mid n \in \omega \} \\ &
			= \{P(1 \oslash (\alpha + \beta)) \oplus \bigoplus_{i=1}^m P_i(1 \oslash (\alpha + \beta_i)) \oplus Q(\alpha) \mid Q \in \alpha[x],\  P \in \alpha[x],\ m \geq 0, \ P_i \in \alpha[x], \\ & \quad  \text{ for every } i \in \{1,\dots, m\},\  \beta_i < \beta,  \text{ for every } i \in \{1,\dots, m\} \} \\ &
			= \alpha [\alpha, \{ 1 \oslash (\alpha + \beta) \mid \beta < \alpha' \}] = \{ P(\alpha) \oslash (\bigotimes_{i=1}^m (\alpha \oplus \beta_i)^{\boxedexp{m_i}}) \mid m \geq 0,\ P \in \alpha[x], \ \beta_i < \alpha',\  \\ & \quad m_i \in \mathbb{N}^*,  \text{ for every } i \in \{1,\dots, m\} \} \text{ is a ring}.
		\end{align*}

		If $\alpha'$ is a limit ordinal, then:
		\begin{align*}
			\alpha^{\omega + \omega \times \alpha'} &= \sup \{ \alpha^{\omega + \omega \times \beta} \mid \beta < \alpha' \} \\ &
			= \{ \bigoplus_{i=1}^m P_i(1 \oslash (\alpha + \beta_i)) \oplus Q(\alpha) \mid Q \in \alpha[x],\ m \geq 0,\ P_i \in \alpha[x], \text{ for every } i \in \{1,\dots, m\},\\ & \quad \beta_i < \alpha', \text{ for every } i \in \{1,\dots, m\} \} \\ &
			= \alpha [\alpha, \{ 1 \oslash (\alpha + \beta) \mid \beta < \alpha' \}] = \{ P(\alpha) \oslash (\bigotimes_{i=1}^m (\alpha \oplus \beta_i)^{\boxedexp{m_i}}) \mid m \geq 0,\ P \in \alpha[x],\ \beta_i < \alpha',\\ & \quad m_i \in \mathbb{N}^*, \text{ for every } i \in \{1,\dots, m\} \} \text{ is a ring}.
		\end{align*}

		We prove the other part of the induction regardless of whether $\alpha'$ is a successor or limit ordinal.
		
		We know from induction that $1 \oslash (\alpha + \beta) = \alpha^{\omega + \omega \times \beta}$, for every $ \beta < \alpha'$.
		Then, any ordinal smaller than $\alpha + \alpha'$ has an inverse in $\alpha^{\omega + \omega \times \alpha'}$.
		Assume towards a contradiction that $\alpha + \alpha'$ has an inverse in $\alpha^{\omega + \omega \times \alpha'}$.

		We know that $\alpha + \alpha' = \alpha \oplus \alpha'$, because $\alpha$ is a group.
		Then, $(\alpha \oplus \alpha') \otimes (P(\alpha) \oslash (\bigotimes_{i=1}^m (\alpha \oplus \beta_i)^{\boxedexp{m_i}})) = 1$, for some $P \in \alpha[x], \beta_i < \alpha', m_i \in \mathbb{N}^*, i \in \{1,\dots, m\}$.
		Thus $(\alpha \oplus \alpha') \otimes P(\alpha) = \bigotimes_{i=1}^m (\alpha \oplus \beta_i)^{\boxedexp{m_i}}$.
		
		Let $Q(x) = (x \oplus \alpha') \otimes P(x)$ and $R(x) = \bigotimes_{i=1}^m (x \oplus \beta_i)^{\boxedexp{m_i}}$, where $Q, R \in \alpha[x]$.

		Since $\alpha' \neq \beta_i$, for every $ i \in \{0,\dots, m\} $, we have that $ Q \neq R $, so $ Q \oplus R \neq 0,\  Q \oplus R \in \alpha[x]$.
		But $(Q \oplus R)(\alpha) = 0$, a contradiction with the fact that $\alpha$ is transcendental.
		Thus $\alpha + \alpha'$ does not have an inverse in $\alpha^{\omega + \omega \times \alpha'}$, from Theorem~\ref{set} e) it follows that $\alpha^{\omega + \omega \times \alpha'} \otimes (\alpha + \alpha') = 1$.
		
		We prove by induction on $n$ that:
		\begin{itemize}
			\item[\tiny$\bullet$] $\alpha^{\omega + \omega \times \alpha' + n} = \{ P(1 \oslash (\alpha + \alpha')) \oplus \bigoplus_{i=1}^m P_i(1 \oslash (\alpha + \beta_i)) \oplus Q(\alpha) \mid Q \in \alpha[x],\ P \in \alpha[x] \text{ of degree at most } n,\ m \geq 0,\ P_i \in \alpha[x], \text{ for every } i \in \{1,\dots, m\},\ \beta_i < \alpha', \text{ for every } i \in \{1,\dots, m\} \}$;
			
			\item[\tiny$\bullet$] $\alpha^{\omega + \omega \times \alpha' + n} \otimes a = \alpha^{\omega + \omega \times \alpha' + n} \times a$, for every $ a < \alpha$; 
			\item[\tiny$\bullet$] $\alpha^{\omega + \omega \times \alpha' + n} = (\alpha^{\omega + \omega \times \alpha'})^{\boxedexp{n+1}}$.
			
		\end{itemize}
		
		For $n = 0$, for the first and third bullet point, we have proved this. The second bullet point will not need a base case.

		Assume it is true for any $m \leq n-1$ and we prove it for $n$.
		We have that
		$\alpha^{\omega + \omega \times \alpha' + n} = \{ \alpha^{\omega + \omega \times \alpha' + n-1} \times a_{n-1} + \dots + \alpha^{\omega + \omega \times \alpha'} \times a_0 + a \mid a_0, \dots, a_{n-1} < \alpha,\ a < \alpha^{\omega + \omega \times \alpha'} \}$.

		Since $\alpha^\beta$ is a group for any ordinal $ \beta$, using Theorem~\ref{set} b) and the induction step, we have that
		$\alpha^{\omega + \omega \times \alpha' + n-1} \times a_{n-1} + \dots + \alpha^{\omega + \omega \times \alpha'} \times a_0 + a = (\alpha^{\omega + \omega \times \alpha' + n-1} \otimes a_{n-1}) \oplus \dots \oplus (\alpha^{\omega + \omega \times \alpha'} \otimes a_0) \oplus a$.
		
		Using that $\alpha^{\omega + \omega \times \alpha' + m} = (1 \oslash (\alpha + \alpha'))^{\boxedexp{m+1}}$ and the characterization of $\alpha^{\omega + \omega \times \alpha'}$, it follows that
		$\alpha^{\omega + \omega \times \alpha' + n} = \{ P(1 \oslash (\alpha + \alpha')) \oplus \bigoplus_{i=1}^m P_i(1 \oslash (\alpha + \beta_i)) \oplus Q(\alpha) \mid Q \in \alpha[x],\ P \in \alpha[x] \text{ of degree at most } n,\ m \geq 0, P_i \in \alpha[x], \text{ for every } i \in \{1,\dots, m\},\ \beta_i < \alpha', \text{ for every } i \in \{1,\dots, m\} \}$.
		
		We prove that $\alpha^{\omega + \omega \times \alpha' + n} \otimes a = \alpha^{\omega + \omega \times \alpha' + n} \times a$, for every $ a < \alpha$ by induction on $a$.
		We know that $\alpha^{\omega + \omega \times \alpha' + n} \otimes a = \mex \{ (\alpha^{\omega + \omega \times \alpha' + n} \otimes a') \oplus (\beta \otimes (a \oplus a')) \mid a' < a, \beta < \alpha^{\omega + \omega \times \alpha' + n} \}$.
		Because $a \oplus a' \in \alpha \setminus \{0\}$,\  $f: \alpha[x] \to \alpha[x],\ f(P)(x) = P(x) \otimes (a \oplus a')$ is bijective and degree-preserving, we have that
		$\{ \beta \otimes (a \oplus a') \mid \beta < \alpha^{\omega + \omega \times \alpha' + n} \} = \alpha^{\omega + \omega \times \alpha' + n}$.
		By induction, from the fact that $\alpha^{\omega + \omega \times \alpha' + n}$ is a group and from Theorem~\ref{set} b), we have that
		$\alpha^{\omega + \omega \times \alpha' + n} \otimes a = \mex \{ (\alpha^{\omega + \omega \times \alpha' + n} \times a') + \beta \mid a' < a, \ \beta < \alpha^{\omega + \omega \times \alpha' + n} \}$
		$= \alpha^{\omega + \omega \times \alpha' + n} \times a$.
		
		Since $\alpha$ is a group, $\alpha^{\omega + \omega \times \alpha' + n}$ is a group. We prove that $(\alpha^{\omega + \omega \times \alpha'}, \alpha^{\omega + \omega \times \alpha' + n-1})$ is the lexicographically least pair of elements from $\alpha^{\omega + \omega \times \alpha' + n}$ with the Nim product greater than or equal to $\alpha^{\omega + \omega \times \alpha' + n}$.
		
		Since $\alpha^{\omega + \omega \times \alpha'} = 1 \oslash (\alpha + \alpha')$ and from the characterization of $\alpha^{\omega + \omega \times \alpha' + n}$ and $\alpha^{\omega + \omega \times \alpha' + n-1}$, it follows that $\alpha^{\omega + \omega \times \alpha'} \otimes \beta \in \alpha^{\omega + \omega \times \alpha' + n}$, for every $ \beta < \alpha^{\omega + \omega \times \alpha' + n-1}$.
		From the characterization of $\alpha^{\omega + \omega \times \alpha' + n}$ and $\alpha^{\omega + \omega \times \alpha'}$ it follows that $\beta \otimes \gamma \in \alpha^{\omega + \omega \times \alpha' + n}$, for every $ \beta < \alpha^{\omega + \omega \times \alpha'},\  \gamma < \alpha^{\omega + \omega \times \alpha' + n}$.
		
		Assume towards a contradiction that $\alpha^{\omega + \omega \times \alpha'} \otimes \alpha^{\omega + \omega \times \alpha' + n-1} = P(\alpha) \oslash ((\bigotimes_{i=1}^m (\alpha \oplus \beta_i)^{\boxedexp{m_i}}) \otimes (\alpha \oplus \alpha')^{\boxedexp{n'}}),$ where $P \in \alpha[x]$,\ $m_i \in \mathbb{N}^*$,\ $\beta_i < \alpha'$, for every $ i \in \{1,\dots, m\}$, $n' \leq n$.
		From induction $\alpha^{\omega + \omega \times \alpha'} \otimes \alpha^{\omega + \omega \times \alpha' + n-1} = (1 \oslash (\alpha + \alpha'))^{\boxedexp{n+1}}$.
		
		The equality becomes $P(\alpha) \otimes (\alpha + \alpha')^{\boxedexp{n - n' + 1}} = \bigotimes_{i=1}^m (\alpha + \beta_i)^{\boxedexp{m_i}}$, which is in contradiction with $\alpha$ transcendental, because $\alpha' \neq \beta_i$, for every $ i \in \{1,\dots, m\}$ and $n - n' + 1 > 0$.
		Thus, the induction is finished.
		
		Therefore:
		\begin{align*}
			\alpha^{\omega + \omega \times \alpha} &= \sup \{ \alpha^{\omega + \omega \times \alpha' + n} \mid \alpha' < \alpha, n \in \omega \} \\ &
			= \{ \bigoplus_{i=1}^m P_i(1 \oslash (\alpha + \beta_i)) \oplus Q(\alpha) \mid Q \in \alpha[x],\ m \geq 0,\ P_i \in \alpha[x],\\ & \quad \text{ for every } i \in \{0,\dots, m\},\ \beta_i < \alpha, \text{ for every } i \in \{1,\dots, m\} \} \\ &
			= \{ P(\alpha) \oslash (\bigotimes_{i=1}^m (\alpha \oplus \beta_i)^{\boxedexp{m_i}}) \mid m \geq 0,\ P \in \alpha[x],\ \beta_i < \alpha,\ m_i \in \mathbb{N}^*, \text{ for every } i \in \{1,\dots, m\} \}.
		\end{align*}

		But $\alpha$ is an algebraically closed field, so $\alpha^{\omega + \omega \times \alpha} = \{ P(\alpha) \oslash Q(\alpha) \mid P,\ Q \in \alpha[x], \ Q \neq 0 \} = \alpha(\alpha)$, so it is a field.
		
		Let $\xi$ be a field ordinal greater than $\alpha$.
		Then $\alpha(\alpha) \subseteq \xi$, so $\alpha^{\omega + \omega \times \alpha} \leq \xi$.
		We have proved that the next field ordinal after $\alpha$ is $\alpha^{\omega + \omega \times \alpha}$.
	\end{proof}

	Next, we treat the case where $\alpha$ is not algebraically closed. This case has a shorter proof, because we already know how ordinals within $\alpha$ add and multiply, due to Theorem~\ref{set} f).

	\begin{proposition} \label{urm-corp-2}
		Let $\alpha$ be a field ordinal that is not transcendental. Then the next field ordinal is $\alpha^n$, where $n$ is the degree of the lexicographically least polynomial from $\alpha[x]$ with no Nim roots in $\alpha$. Moreover, $ \alpha^n = \alpha(\alpha) $.
	\end{proposition}
	
	\begin{proof}
		Let $f$ be the lexicographically least polynomial from $\alpha[x]$ with no Nim roots in $\alpha$. It follows that $f$ is monic. Let:
		$$f(x) = x^{\boxedexp{n}} \oplus a_{n-1} \otimes x^{\boxedexp{n-1}} \oplus \dots \oplus a_0 \text{, where }  a_0, \dots, a_{n-1} < \alpha.$$
		From Theorem~\ref{set} g) we have that $f(\alpha) = 0$.

		But $\alpha^n = \{ \alpha^{n-1} \times b_{n-1} + \dots + b_0 \mid b_0, \dots, b_{n-1} < \alpha \}$
		$= \{ \alpha^{\boxedexp{n-1}} \otimes b_{n-1} \oplus \dots \oplus b_0 \mid b_0, \dots, b_{n-1} < \alpha \}$
		$= \alpha(\alpha)$, from Theorem~\ref{set} f).
		So, $\alpha^n$ is a field.

		Let $\xi$ be a field ordinal greater than $\alpha$.
		Then $\alpha(\alpha) \subseteq \xi$, so $\alpha^n \leq \xi$.
		We have proved that the next field ordinal after $\alpha$ is $\alpha^n$.
	\end{proof}
	
	\begin{proposition} \label{corpuri}
		Let $\alpha,  \beta$ field ordinals, $\alpha \leq \beta$. Then there is an ordinal $\gamma$ such that $\beta = \alpha^{\gamma}$.
	\end{proposition}
	
	\begin{proof}
		It follows from Proposition~\ref{urm-corp-1}, Proposition~\ref{urm-corp-2} and from the fact that the supremum of a set of field ordinals is also a field ordinal. 
	\end{proof}
	
	\begin{proposition} \label{urm-corp-dupa-inch-alg}
		If $\alpha$ is a transcendental ordinal, the next field ordinal is $\alpha^{\alpha}$. Moreover, $ \alpha^{\alpha} = \alpha(\alpha) $.
	\end{proposition}
	
	\begin{proof}
	
	We know that the first transcendental ordinal is $\omega^{\omega^{\omega}}$.
 	From Proposition~\ref{corpuri}, it follows that $\alpha = \omega^{\omega^{\omega} \times \beta}$ for some ordinal $\beta$.
	
	Thus, $\alpha = \omega^{\omega^{\omega} \times \beta}=\omega^{\delta}$, where $\delta = \omega^{\omega} \times \beta$.  It follows that $\alpha$ and $\delta$ are limit ordinals.

	We know that the next field is $\alpha^{\omega + \omega \times \alpha}$. But $\omega + \omega \times \alpha = \omega \times (1+ \alpha) = \omega \times \alpha = \omega \times \omega^{\delta} = \omega^{1+\delta} = \omega^{\delta} = \alpha $, because $\alpha$ and $\delta$ are limit ordinals. Thus the next field is $\alpha^{\alpha}$.
	
\end{proof}

	\section{The next quadratically closed field ordinal} \label{sec: 3}
	
	\subsection{About the next transcendental} \label{subsec: 3.1}
	
	Proposition~\ref{urm-corp-1} and Proposition~\ref{urm-corp-2} imply that $\alpha(\alpha)$ is an ordinal, for any field ordinal $\alpha$. This result is the key to the next proposition.

	\begin{proposition} \label{inch-alg-e-ordinal}
		Let $u$ be a field ordinal. Then the algebraic closure of $u$ in $\ordinali$ is an ordinal and is an algebraic closure of $u$.
	\end{proposition}
	
	\begin{proof}
		Let $F \subset \ordinali$ be the algebraic closure of $u$ in $\ordinali$.
		Assume towards a contradiction that  $F$ is not an ordinal.
		
		Let $\alpha$ be a field ordinal such that $u \subset \alpha \subset F$. If $\alpha$ is transcendental, then $F\subset \alpha$ so $F=\alpha$. Since $F$ is not an ordinal, $\alpha$ is not transcendental. Then from Theorem~\ref{set} g), it follows that $\alpha$ is a Nim root of the lexicographically least irreducible polynomial over $\alpha$, so $\alpha \in F$. 
		
		For all ordinals $\alpha$, we define the ordinals $u_{\alpha}$ such that $u_0 =u$, $u_{\beta + 1}= u_{\beta}(u_{\beta})$ and $u_{\alpha}= \sup\{u_{\beta} \mid \beta < \alpha \}$, for $\alpha$ limit ordinal. 
		We prove by induction that $u_{\alpha} \geq \alpha $. For $\alpha \in \{0,1,2\} $ is obvious. For $\beta>1$, since $u_{\beta+1} \geq u_{\beta}^2 \geq \beta^2$, from Proposition~\ref{corpuri}, it follows that $u_{\beta+1} \geq \beta +1$. For $\alpha$ limit ordinal, we have that $u_{\alpha}= \sup\{u_{\beta} \mid \beta < \alpha \} \geq \sup\{\beta \mid \beta < \alpha\} = \alpha $.
		
		We prove by induction that $u_{\alpha} \subset F$ from which it follows that $u_{\alpha} \in F$, $\alpha \in F$. 
		
		We know that $u_0 \subset F$. 
		
		If $\alpha = \beta + 1$ and $u_{\beta} \subset F $, then $u_{\beta} \in F $, so $u_{\beta}(u_{\beta}) \subset F$, so $u_{\beta+1} \subset F$. 
		
		If $\alpha$ is a limit ordinal and $u_{\beta} \subset F $, for all $\beta < \alpha$, then $\sup \{u_{\beta} \mid \beta < \alpha \} \subset F $, so $u_{\alpha} \subset F$. 
		
		So $F = \ordinali$. But $F$ is a set and $\ordinali$ is a proper class, contradiction.
		
		We prove that $F$ is an algebraic closure of $u$.
		We have to prove that $F$ is algebraically closed. If it is not, then from Theorem~\ref{set} g), it follows that $F$ is a Nim root of the lexicographically least irreducible polynomial over $F$. So $F$ is algebraic over $F$. Then $u \subset F(F)$ is an algebraic extension, so $F$ is algebraic over $u$, so $F \in F$, contradiction.
	\end{proof}
	
	It follows that the next transcendental ordinal after $\tau$ is $\overline{\tau(\tau)}$ and that $\ordinali$ is an algebraically closed Field.
	
	An analogous result can be proven about the $n$-closure of a field ordinal.
	\begin{proposition} \label{n-inch-e-ordinal}
		Let $u$ be a field ordinal, $n \in \mathbb{N}^*$. Then the $n$-closure of $u$ in $\ordinali$ (which is an $n$-closure of $u$) is an ordinal.
	\end{proposition}
	
	\begin{proof}
		Let $F \subset \ordinali$ be the $n$-closure of $u$ in $\ordinali$.
		Assume towards a contradiction that $F$ is not an ordinal.
		
		Let $\alpha$ be a field ordinal such that $u \subset \alpha \subset F$. Since $F$ is not an ordinal, $\alpha$ is not $n$-closed. Then, from Theorem~\ref{set} g), it follows that $\alpha$ is a root of the  smallest lexicographically irreducible polynomial over $\alpha$ which has degree at most $n$, thus $\alpha \in F$. 
		
		For each ordinal $\alpha$, we recursively define the ordinals $u_{\alpha}$ such that $u_0 =u$, $u_{\beta + 1}= u_{\beta}(u_{\beta})$ and $u_{\alpha}= \sup\{u_{\beta} \mid \beta < \alpha \}$, for $\alpha$ a limit ordinal. We have that $u_{\alpha} \geq \alpha $, for every ordinal $\alpha$.

		We prove by induction on $\alpha$ that $u_{\alpha} \subset F$ from which it follows that $u_{\alpha} \in F$. 
		
		We know that $u_0 \subset F$. 
		
		If $\alpha = \beta + 1$ and $u_{\beta} \subset F $, then $u_{\beta} \in F $, so $u_{\beta}(u_{\beta}) \subset F$, hence $u_{\beta+1} \subset F$. 
		
		If $\alpha$ is a limit ordinal and $u_{\beta} \subset F $, for every $\beta < \alpha$, then $\sup \{u_{\beta} \mid \beta < \alpha \} \subset F $, thus $u_{\alpha} \subset F$. 
		
		Therefore $F = \ordinali$. However, $F$ is the $n$-closure of an ordinal, hence a set, while $\ordinali$ is a proper class, contradiction.
	\end{proof}
	
	\begin{lemma} \label{reuniune-corp-n-inchis}
		Let $u$ be a field ordinal. Then $\overline{u} = \sup\{ \overline{u}^n \mid n \in \mathbb{N}^*  \}$.
	\end{lemma}
	
	\begin{proof}
		From Proposition~\ref{inch-alg-e-ordinal} and Proposition~\ref{n-inch-e-ordinal}, it follows that $\overline{u}$, $\overline{u}^n$, for all $n$ are ordinals.
		It is obvious that $ u= \overline{u}^1 \subset \overline{u}^2 \subset \dots \subset \overline{u}^n \subset \dots \subset \overline{u}  $, so $\sup\{ \overline{u}^n \mid n \in \mathbb{N}^*  \} \subset \overline{u}$. 
		
		We prove that $F =\sup\{ \overline{u}^n \mid n \in \mathbb{N}^*  \}$ is an algebraically closed field. It is obvious that it is a field. Let $f \in F[X] $. There is an integer $m \in \omega$ such that $m > \deg(f) $ and all coefficients of $f$ are in $\overline{u}^m$. Then all roots of $f$ are in $\overline{u}^m$, so in $F$. It follows that $F = \overline{u}$.
	\end{proof}
	
	Therefore, in order to find the next transcendental ordinal after $\tau$, we need to find $ \overline{\tau(\tau)}^n$, for $n \in \mathbb{N}^*$. In Subsection~\ref{sec: 3.4}, we find $ \overline{\tau(\tau)}^2$. Before making the first steps into the search of $ \overline{\tau(\tau)}^2$, we give an example of another transcendental ordinal.

	\begin{definition} \label{def-s-x}
		Let $x$ be an ordinal. We define $s(x)$ to be the minimal field ordinal that is strictly greater than $x$.
	\end{definition}
	
	The above is well-defined because from Propositions~\ref{urm-corp-1} and~\ref{urm-corp-2} and from the fact that the supremum of a set of field ordinals is a field ordinal, it follows that there exist arbitrarily large field ordinals.
	
	\begin{lemma} \label{lema-s-x}
		Let $x$ be an ordinal. Then $x < s(x) \leq x^x$.
	\end{lemma}
	
	\begin{proof}
		Let $\beta = \sup \{ \alpha \leq x \mid \alpha \text{ field} \}$. Since the supremum of a set of field ordinals is a field ordinal, it follows that $ \beta \in \{ \alpha \leq x \mid \alpha \text{ field} \} $. Then $s(x) = s(\beta)$, because in the set $\{ \alpha \mid \beta < \alpha \leq x  \} $ there are no field ordinals. 
		From Propositions~\ref{urm-corp-dupa-inch-alg} and~\ref{urm-corp-2} it follows that $s(\beta) \leq \beta^\beta $, so $\beta \leq x < s(x) = s(\beta) \leq \beta^\beta \leq x^x$.
	\end{proof}
	
	\begin{proposition} \label{omega-1-transc}
		The ordinal $\omega_1$ is transcendental.
	\end{proposition}
	
	\begin{proof}
		We prove that $\omega_1$ is a field.
		Let $x \leq y < \omega_1$. Then $y < s(y) \leq y^y < \omega_1$, from Lemma~\ref{lema-s-x}. Since $s(y)$ is a field, it follows that $ x \oplus y, \ x \otimes y, \ 1 \oslash y < s(y) < \omega_1 $, so $\omega_1$ is a field.
		
		Let $p \in \omega_1[X]$. We denote by $a_0, a_1, \dots, a_n$ the coefficients of $p$ and $a = \max \{ a_0, \dots, a_n \}$. From Lemma~\ref{lema-s-x}, we have that $a< s(a) \leq a^a$. It follows that $s(a)< \omega_1$. From Proposition~\ref{inch-alg-e-ordinal} it follows that the algebraic closure of $s(a)$ is an ordinal, which we denote by $\alpha$.
		From Lemma~\ref{corp-num} it follows that $\alpha < \omega_1$. Since $\alpha $ is algebraically closed and $p \in \alpha[X]$, it follows that the roots of $p$ are in $\alpha$, and thus also in $\omega_1$.
	\end{proof}
	
	The following proposition states that there are uncountably many transcendental ordinals between $\tau$ and $\omega_1$.
	
	\begin{proposition} \label{multi-transc}
		We have that $ \otype \{ \alpha < \omega_1 \mid \alpha \text{ transcendental ordinal}  \} = \omega_1 $.
	\end{proposition}
	
	\begin{proof}
		Let $\gamma = \otype \{ \alpha < \omega_1 \mid \alpha \text{ transcendental ordinal}  \} \leq \omega_1 $. 
		
		Let $\beta = \sup \{ \alpha < \omega_1 \mid \alpha \text{ transcendental ordinal}  \} $. If $\beta < \omega_1$, we have that $\beta^\beta < \omega_1$. From Proposition~\ref{inch-alg-e-ordinal} and Lemma~\ref{corp-num} it follows that the algebraic closure of $\beta^\beta$ is an ordinal smaller than $\omega_1$, a contradiction with the choice of $\beta$. So $\beta = \omega_1$. If $\gamma < \omega_1$, since $\omega_1 = \sup \{ \alpha < \omega_1 \mid \alpha \text{ transcendental ordinal}  \}$ and a countable union of countable sets is countable, we arrive at a contradiction. Therefore $\gamma = \omega_1$.
	\end{proof}
	
	\subsection{The perfect closure of the next field after a transcendental}
	
	Before finding the quadratic closure of $\tau(\tau)$, it is useful to know its perfect closure. It turns out that finding its perfect closure is much easier and we can also find some identities, just like in Proposition~\ref{urm-corp-1}. 
	
	\begin{proposition} \label{inch-perfecta-ord-corp-alg-inchis}
		Let $\alpha $ be an algebraically closed field ordinal. Then the perfect closure of $\alpha^\alpha$ is $\alpha^{\alpha \times \omega}$. Moreover, we have that $(\alpha^\alpha)\boxedexp{2} = \alpha$, $\alpha^{\alpha \times 2^i}$ is a field, $\alpha^{\alpha \times 2^i} = \alpha ^{\alpha \times 2^{i-1}}(\alpha ^{\alpha \times 2^{i-1}})$, and $ (\alpha^{\alpha \times 2^i})\boxedexp{2} = \alpha^{\alpha \times 2^{i-1}}$, for every $i \in \mathbb{N}^*$.
	\end{proposition}
	
	\begin{proof}
		
		From Proposition~\ref{urm-corp-dupa-inch-alg}, we know that $\alpha^{\alpha} = \alpha(\alpha) $ is the next field. Since $\alpha$ is algebraically closed, it follows that for every $\beta < \alpha$, the polynomial $X\boxedexp{2} \oplus \beta$ has roots in $\alpha$, and thus also in $\alpha^\alpha$.
		
		If $X\boxedexp{2} - \alpha$ were to have the root $\eta$ in $\alpha^\alpha$, then $\eta = P(\alpha) \oslash Q(\alpha) $, where $P, \ Q \in \alpha[X]$, $Q \neq 0$. We have that:
		
		$$ (P(\alpha) \oslash Q(\alpha))\boxedexp{2} = \alpha. $$
		
		Then, $ P(\alpha)\boxedexp{2} = \alpha \otimes Q(\alpha)\boxedexp{2} $, but since $\alpha$ is transcendental over $\alpha$, it follows that $P(X)^2 = X Q(X)^2$, but the left-hand side has an even degree, while the right-hand side has an odd degree. Thus, $X\boxedexp{2} \oplus \alpha$ has no roots in $\alpha^\alpha$. 
		
		From Theorem~\ref{set} g) it follows that $(\alpha^\alpha)\boxedexp{2} = \alpha$, and from Theorem~\ref{urm-corp-2} it follows that $\alpha^{\alpha \times 2}  =  \alpha^\alpha(\alpha^\alpha)$ is the next field after $\alpha^\alpha$.

		Next, we prove by induction on $i \in \mathbb{N}^*$ that:
		\begin{enumerate} [a)]
			\item The ordinal $\alpha^{\alpha \times 2^i}$ is a field and $\alpha^{\alpha \times 2^i} = \alpha ^{\alpha \times 2^{i-1}}(\alpha ^{\alpha \times 2^{i-1}})$;
			\item For every $\beta < \alpha^{\alpha \times 2^{i-1}}$, the polynomial $X\boxedexp{2} \oplus \beta $ has roots in $ \alpha^{\alpha \times 2^i} $;
			\item $ (\alpha^{\alpha \times 2^i})\boxedexp{2} = \alpha^{\alpha \times 2^{i-1}}$.
		\end{enumerate}
		
		For $i=1$, we proved that $\alpha^{\alpha \times 2} = \alpha^\alpha(\alpha^\alpha)$ is a field. Let $\beta < \alpha^\alpha$. Then $\beta = P(\alpha) \oslash Q(\alpha) $, where $P, \ Q \in \alpha[X]$, $Q \neq 0$. Since $P(X) = a_0 \oplus \dots \oplus a_n\otimes X\boxedexp{n}$, where $a_0, \dots , a_n \in \alpha$, and $\alpha$ is algebraically closed, it follows that there exist $b_0, \dots , b_n \in \alpha $ such that $a_0 = b_0\boxedexp{2}, \dots , \ a_n = b_n\boxedexp{2}$, so $P(\alpha)= b_0\boxedexp{2} \oplus \dots \oplus b_n\boxedexp{2}\otimes (\alpha^\alpha)\boxedexp{2n} = (b_0 \oplus \dots \oplus b_n\otimes (\alpha^\alpha)\boxedexp{n})\boxedexp{2} = c\boxedexp{2}$, where $ c \in \alpha^{\alpha \times 2 }$. Analogously, $Q(\alpha) = d\boxedexp{2}$, where $ d \in \alpha^{\alpha \times 2 }$. So $X\boxedexp{2} \oplus \beta $ has the root $c \oslash d$ in $ \alpha^{\alpha \times 2} $. 
		
		Assume that there exists $\eta \in \alpha^{\alpha \times 2}$ such that $ \eta \boxedexp{2} = \alpha^\alpha $. Then there exist $\eta_1, \eta_2 \in \alpha^\alpha$ such that $\eta = \eta_1 \oplus \eta_2 \otimes \alpha^\alpha$. We have that:
		
		$$ \alpha^\alpha = \eta\boxedexp{2} = \eta_1\boxedexp{2} \oplus \eta_2\boxedexp{2}\otimes \alpha \in  \alpha^\alpha \text{, contradiction.}$$
		
		By Theorem~\ref{set} g), it follows that $ (\alpha^{\alpha \times 2})\boxedexp{2} = \alpha^\alpha $.
		
		To prove the induction step, we assume it is true for $i-1$ and prove it for $i \geq 2$. From the induction hypothesis, we know that $\alpha ^{\alpha \times 2^{i-1}}$ is a field in which the lexicographically least irreducible polynomial is $X\boxedexp{2} \oplus \alpha^{\alpha \times 2^{i-2}} $. We have that $\alpha^{\alpha \times 2^i} = \alpha ^{\alpha \times 2^{i-1}}(\alpha ^{\alpha \times 2^{i-1}})$ is the next field after $\alpha ^{\alpha \times 2^{i-1}}$. 
		
		Let $\beta < \alpha ^{\alpha \times 2^{i-1}}$. Then there exist $\beta_1, \ \beta_2 \in \alpha ^{\alpha \times 2^{i-2}}$ such that $\beta = \beta_1 \oplus \beta_2 \otimes \alpha ^{\alpha \times 2^{i-2}}$. 
		From the induction hypothesis, there exist $\beta'_1, \ \beta'_2 \in \alpha ^{\alpha \times 2^{i-1}}$
		such that $\beta_1 = {\beta'_1}\boxedexp{2}, \ \beta_2 = {\beta'_2}\boxedexp{2}$. Then $\beta = (\beta'_1 \oplus \beta'_2 \otimes \alpha ^{\alpha \times 2^{i-1}})\boxedexp{2}$. Since $ \beta'_1 \oplus \beta'_2 \otimes \alpha ^{\alpha \times 2^{i-1}} \in \alpha^{\alpha \times 2^i}$, it follows that the polynomial $X\boxedexp{2} \oplus \beta $ has roots in $ \alpha^{\alpha \times 2^i} $.
		
		Assume there exists $\eta \in \alpha^{\alpha \times 2^i} $ such that $ \eta \boxedexp{2} = \alpha^{\alpha \times 2^{i-1}} $. Then there exist $\eta_1, \ \eta_2 \in \alpha ^{\alpha \times 2^{i-1}}$ such that $\eta = \eta_1 \oplus \eta_2 \otimes \alpha ^{\alpha \times 2^{i-1}}$. We have that:
		
		$$ \alpha^{\alpha \times 2^{i-1}} = \eta \boxedexp{2} = {\eta_1}\boxedexp{2} \oplus {\eta_2}\boxedexp{2} \otimes \alpha ^{\alpha \times 2^{i-2}} \in \alpha^{\alpha \times 2^{i-1}} \text{, contradiction.}$$
		
		By Theorem~\ref{set} g), it follows that $ (\alpha^{\alpha \times 2^i})\boxedexp{2} = \alpha^{\alpha \times 2^{i-1}}$.
		
		We have obtained from induction that for $ i \in \mathbb{N}^* $, the $i$-th field after $\alpha^\alpha$ is $\alpha^{\alpha \times 2^{i}}$. Then $\sup\{\alpha^{\alpha \times 2^{i}} \mid i \in \mathbb{N}^* \} = \alpha^{\alpha \times \omega}$ is a field. Let $\xi \in \alpha^{\alpha \times \omega} $. Then there exists $ i \in \mathbb{N}^* $ such that $ \xi \in \alpha^{\alpha \times 2^{i}} $. From ii) of the induction, it follows that $X\boxedexp{2} - \xi$ has a root in $ \alpha^{\alpha \times 2^{i+1}} \subset \alpha^{\alpha \times \omega}  $. So $\alpha^{\alpha \times \omega}$ is a perfect field. 
		
		Let $F \subset \ordinali$ be a perfect field such that $\alpha^\alpha \subset F$. Then, since every element of $F$ has its square root in $F$, it follows that $\alpha^\alpha \in F$, because $ \alpha \in F$, so $ \alpha^\alpha(\alpha^\alpha ) \subset F$. Then $ \alpha^{\alpha \times 2 } \subset F$ and $ \alpha^{\alpha \times 2 } \in F$. Inductively it follows that for any $n \in \mathbb{N}$, $ \alpha^{\alpha \times 2^n } \subset F$ and $ \alpha^{\alpha \times 2^n } \in F$. So $\alpha ^ {\alpha \times \omega } \subset F$. Thus, $\alpha ^ {\alpha \times \omega } $ is the perfect closure of $\alpha^\alpha$.

	\end{proof}
	
	\begin{remark} \label{obs-perfect-pana-la-transcendent2}
		From Proposition~\ref{inch-perfecta-ord-corp-alg-inchis}, it follows that the perfect closure of $\omega^{\omega^{\omega^{\omega}}}$ is $ \tau^{\tau \times \omega} = \omega^{\omega^{\omega^{\omega}+1}}$. From Proposition~\ref{extindere-alg-corp-perfect}, it follows that any algebraic extension of $\omega^{\omega^{\omega^{\omega}+1}}$ is perfect, so any field ordinal greater than or equal to $\omega^{\omega^{\omega^{\omega}+1}}$ and smaller than the second transcendental ordinal is perfect.
	\end{remark}

	\begin{proposition} \label{rel-inch-perf}
		Let $t$ be a transcendental ordinal. We have that $ (t^{t \times \omega}) \boxedexp{2} \oplus t^{t \times \omega} \oplus t =0 $. In particular, $t^{t \times \omega}$ is not quadratically closed.
	\end{proposition}
	
	\begin{proof}
		Since $t^{t \times \omega}$ is a perfect field and $t $ is an algebraically closed field, any polynomial smaller (in the lexicographical sense) than $X \boxedexp{2} \oplus X \oplus t$ has roots in $t^{t \times \omega}$. By Theorem~\ref{set} g), it is enough to prove that $X \boxedexp{2} \oplus X \oplus t$ is irreducible over $ t^{t \times \omega} $. Being a polynomial of degree 2, it is enough to prove that there is no $x \in t^{t \times \omega} $ such that $x \boxedexp{2} \oplus x \oplus t = 0$.

		We prove by induction on $i$ that there is no $x \in t^{t \times 2^i} $ such that $x \boxedexp{2} \oplus x \oplus t = 0$.

		Assume there exists $x \in t(t) $ such that $x \boxedexp{2} \oplus x \oplus t = 0$. Then, $x = p(t) \oslash q(t)$, where $p, q \in t[X]$, $q$ monic, $(p,q)=1$. So $p \boxedexp{2} \oplus p \otimes q \oplus X \otimes q \boxedexp{2} = 0$, from which it follows that $ q \mid p \boxedexp{2}$, so $q=1$. It follows that $ p \boxedexp{2} \oplus p \oplus X =0 $. If $p$ is constant, then $\deg(p \boxedexp{2} \oplus p \oplus X ) =1 $, contradiction. If $\deg (p) \geq 1$, then $\deg(p \boxedexp{2} \oplus p \oplus X ) \geq 2$, contradiction. Thus, there is no $x \in t(t) $ such that $x \boxedexp{2} \oplus x \oplus t = 0$. 
		
		Assume there exists $x \in t^{t \times 2^{i+1}}  $ such that $x \boxedexp{2} \oplus x \oplus t = 0$. Then $x = t ^ {t \times 2^i} \otimes \alpha \oplus \beta$, where $\alpha, \beta < t^{t \times 2^i}$. So $ t^{t \times 2^i} \otimes \alpha \oplus (\beta \oplus \beta \boxedexp{2} \oplus (t^{t \times 2^i})\boxedexp{2} \otimes \alpha \boxedexp{2} \oplus t  ) = 0 $. Since $(t^{t \times 2^i}) \boxedexp{2} < t^{t \times 2^i} $, it follows that $ \beta \oplus \beta \boxedexp{2} \oplus (t^{t \times 2^i})\boxedexp{2} \otimes \alpha \boxedexp{2} \oplus t     < t^{t \times 2^i} $, so $\alpha = 0$, thus $x < t^{t \times 2^i}$, contradiction with the induction hypothesis.

	\end{proof}

	\begin{remark} 
		We have that $ (\tau^{\tau \times \omega}) \boxedexp{2} \oplus \tau^{\tau \times \omega} \oplus \tau =0 $.
	\end{remark}

	\subsection{The quadratic closure of a perfect field ordinal}
	
	The results in this subsection were published by H. Lenstra in \cite{nim-mult}.
	
	\begin{definition} \label{def-P}
		For an ordinal $x$, we denote $P(x) = x\boxedexp{2} \oplus x$. We denote by $x^*$ the smallest ordinal $y$ such that $P(y)=x$. 
	\end{definition}

	\begin{remark}
		In Definition~\ref{def-P} we know that $\{y \mid P(y)=x\}$ is non-empty, since there exist roots of the polynomial $P(X) \oplus x$ in $\ordinali$.
	\end{remark}
	
	\begin{remark}
		For any ordinal $x$, the roots of the polynomial $P(X) \oplus x$ are $x^*$ and $x^*\oplus 1$. We have that $x^* \oplus 1>x^*$. 
		
	\end{remark}
	
	\begin{proposition} \label{lipseste-1-la-star}
		Let $x$ be an ordinal. We write $x = 2^{\beta_1}+\dots + 2^{\beta_n}$, where $\beta_1 > \dots > \beta_n $. Then $\beta_n = 0$ if and only if $x \oplus 1 < x$.
	\end{proposition}

	\begin{proof}
		If $\beta_n = 0$, then $x \oplus 1 = 2^{\beta_1} \oplus \dots \oplus 2^{\beta_n} \oplus 1 = 2^{\beta_1} \oplus \dots \oplus 2^{\beta_{n-1}} = 2^{\beta_1}+\dots + 2^{\beta_{n-1}} < x$.
		
		If $\beta_n > 0$, then $x \oplus 1 = 2^{\beta_1} \oplus \dots \oplus 2^{\beta_n} \oplus 1 = 2^{\beta_1}+\dots + 2^{\beta_{n}} +2^0 = x+1 >x$.
	\end{proof}
	
	\begin{proposition}
		For any ordinals $x$, $y$, we have that $(x \oplus y)^*= x^* \oplus y^*$.
	\end{proposition}
	
	\begin{proof}
		Since $ x^*{\boxedexp{2}} \oplus x^* =x $ and $ y^*{\boxedexp{2}} \oplus y^* =y $, it follows that $ (x^* \oplus y^*){\boxedexp{2}} \oplus (x^* \oplus y^*) =x \oplus y $. So $(x \oplus y)^* = x^* \oplus y^*$ or $(x \oplus y)^* = x^* \oplus y^* \oplus 1$.
		
		We write $x^* = 2^{\beta_1}+\dots + 2^{\beta_n}$, where $\beta_1 > \dots > \beta_n $, and $y^* = 2^{\gamma_1}+\dots + 2^{\gamma_m}$, where $\gamma_1 > \dots > \gamma_m $. Then, from Proposition~\ref{lipseste-1-la-star}, we have that $\beta_n>0$, $\gamma_m>0$. 
		
		We have that $x^* \oplus y^* = 2^{\eta_1} + \dots + 2^{\eta_p}$, where $\eta_1>\dots>\eta_p$, and $\{ \eta_1, \dots, \eta_p \} \subset \{ \beta_1, \dots, \beta_n, \gamma_1, \dots, \gamma_m \} $. It follows that $\eta_p > 0$. From Proposition~\ref{lipseste-1-la-star}, it follows that $ x^* \oplus y^* \oplus 1 > x^* \oplus y^* $. Thus $(x \oplus y)^* = x^* \oplus y^*$.
	\end{proof}

	\begin{definition}
		For an ordinal $u$, we denote $P[u] = \{P(x) \mid x \in u  \}$.
	\end{definition}
	
	\begin{remark}
		For a field ordinal $u$, $P[u]$ is an additive subgroup of $u$.
	\end{remark}
	
	\begin{proposition} \label{echiv-2-inchis}
		Let $u$ be an ordinal. Then $u$ is a quadratically closed field if and only if $u$ is a perfect field and $P[u] = u$.
	\end{proposition}
	
	\begin{proof}
		The implication from left to right is obvious. For the other implication, consider an arbitrary monic polynomial of degree 2 over $u$, let this be $f = X\boxedexp{2} \oplus a \otimes X \oplus b$. 
		
		If $a = 0$, then $f$ has a root in $u$, because $u$ is perfect.
		
		If $a \neq 0$, we denote by $c$ the inverse of $a$ in $u$. Since $P[u] = u$, it follows that there exists $y \in u$ such that $P[y]=b \otimes c \boxedexp{2}$. Then $a \otimes y$ is a root of $f$, because $a\boxedexp{2} \otimes y\boxedexp{2} \oplus a\boxedexp{2} \otimes y \oplus b =0$.
	\end{proof}
	
	\begin{lemma} \label{lema-u-x-star}
		Let $u$ be an ordinal. If $u$ is a perfect field, but not quadratically closed, then $u = x^*$, where $x$ is the smallest ordinal in $u \setminus P[u]$.
	\end{lemma}
	
	\begin{proof}
		From Proposition~\ref{echiv-2-inchis}, we have that $u \setminus P[u] \neq \emptyset$, so let $x$ be the minimum of this set. Then the polynomial $X\boxedexp{2} \oplus X \oplus x \in u[X]$ has no roots in $u$ and, being of degree 2, it follows that it is irreducible over $u$. Since $u$ is perfect, any polynomial of the form $X\boxedexp{2} \oplus a \in u[X]$ is reducible. Thus, from the minimality of $x$, it follows that the smallest irreducible polynomial over $u$ is $X\boxedexp{2} \oplus X \oplus x$. From Theorem~\ref{set} g) it follows that $u\boxedexp{2} \oplus u = x$. If $u \oplus 1 < u$, then $u \oplus 1 \in u$, and $ (u \oplus 1) \oplus 1 = u \notin u$, contradicting the fact that $u$ is a group. Thus $u = x^*$.
	\end{proof}

	\begin{definition}
		For a nonzero ordinal $u$, we define:
		\begin{align*}
			L(u) &= \{ \lambda \in u \setminus \{0\} \mid \text{ there exists no } \beta \in u \text{ such that } \lambda \oplus P(\beta) \\ &  \text{ can be written as a finite Nim sum of ordinals smaller than } \lambda \}.
		\end{align*}
		
	\end{definition}

	\begin{proposition} \label{minim-L}
		Let $u$ be a field ordinal such that $P[u] \neq u$. Let $ x = \min(u \setminus P[u])$. Then $x = \min L(u)$.
	\end{proposition}
	
	\begin{proof}
		We prove that $x \in L(u)$. Assume towards a contradiction that $x \notin L(u) $, then there exist $\beta \in u$, and $ \lambda_1, \dots , \lambda_n < x $ such that $x \oplus P(\beta) = \lambda_1 \oplus \dots \oplus \lambda_n$. Since $ x = \min(u \setminus P[u])$, it follows that 
		$ \lambda_1, \dots , \lambda_n \in P[u]$. Since $P[u]$ is an additive subgroup of $u$, it follows that $ x \in P[u] $, a contradiction.
		
		Let $y \in L(u)$. Assume towards a contradiction that $x>y$. Then $y \in P[u]$. For $\beta = y^* \in u$, $y \oplus P(\beta) = 0$, contradicting the fact that $ y \in L(u)$.
	\end{proof}
	
	\begin{proposition} \label{L-nevid}
		Let $u$ be a field ordinal. Then $L(u) \subset u \setminus P[u]$.
	\end{proposition}
	
	\begin{proof}
		Let $\lambda \in P[u]$. Then, clearly, $\lambda \notin P[u]$.
	\end{proof}

	\begin{corollary} \label{cor-L-nevid}
		Let $u$ be a perfect field ordinal. Then $u$ is quadratically closed iff $L(u) = \emptyset$.
	\end{corollary}
	\begin{lemma} \label{def-alternativa-Lu}
		Let $u$ be an ordinal. Then:
		\begin{align*}
			L(u) = \{ \lambda \in u\setminus \{0\}  \mid \lambda \text{ is a group and there exists no } \beta \in u, \ \lambda_1 < \lambda \text{ such that }  P(\beta) = \lambda \oplus \lambda_1 \}.
		\end{align*} 
		
	\end{lemma}
	
	\begin{proof}
		Let $A = \{ \lambda \in u\setminus \{0\}  \mid \lambda \text{ is a group and there exists no } \beta \in u, \ \lambda_1 < \lambda \text{ such that }  P(\beta) = \lambda \oplus \lambda_1 \}$.
		
		Let $\lambda \in L(u)$. If $\lambda$ is not a group, and we know that $\lambda > 0$,  then by Theorem~\ref{set} a), $\lambda$ is written as the sum of two ordinals smaller than it, so for $\beta =0$, we have that $\lambda \oplus P(\beta)$ can be written as a finite Nim sum of ordinals smaller than $\lambda$, a contradiction. If there exist $\beta \in u, \ \lambda_1 < \lambda $ such that $ P(\beta) = \lambda \oplus \lambda_1 $, then $ \lambda \notin L(u) $, a contradiction. Thus $L(u) \subset A $.
		
		Let $\lambda \in A$. If there exist $\beta \in u$, and $ \lambda_1, \dots , \lambda_n < \lambda $ such that $\lambda \oplus P(\beta) = \lambda_1 \oplus \dots \oplus \lambda_n$, then $P(\beta) = \lambda \oplus (\lambda_1 \oplus \dots \oplus \lambda_n)$ and $\lambda_1 \oplus \dots \oplus \lambda_n < \lambda$, because $\lambda$ is a group (in particular, $\lambda>0$), contradicting $\lambda \in A$. So $A \subset L(u)$.
		
	\end{proof}
	
	\begin{theorem} \label{claim-inch-patr}
		Let $u$ be a perfect field ordinal. 
		For each ordinal $\alpha$, we recursively define the ordinals $u_{\alpha}$ such that $u_0 =u$, $u_{\beta + 1}= u_{\beta}(u_{\beta})$ and $u_{\alpha}= \sup\{u_{\beta} \mid \beta < \alpha \}$, for a limit ordinal $\alpha$. 
		
		Let $v := \otype(L(u))$. Suppose that $v \neq 0 $. Let $\varepsilon$ be the smallest $\varepsilon$-number strictly greater than $v$. Then the quadratic closure of $u$ is $u^{\varepsilon} = u_\varepsilon$.
		
	\end{theorem}

	\begin{proof}
		We know from Proposition~\ref{n-inch-e-ordinal} that the quadratic closure of $u$ is also an ordinal, let this be $w$. 
		
		Assume towards a contradiction that $w \neq u_{\alpha}$, for any ordinal $ \alpha $. We prove by induction that $u_{\alpha} \subset w $, for every $\alpha$. Clearly, $u_0 \subset w$. If $u_{\beta} \subset w $, since $w \neq u_{\beta}$ and $w$ is an ordinal, then $u_{\beta} \in w$, so $u_{\beta + 1} \subset w$. If $\alpha$ is a limit ordinal and $u_{\beta} \subset w$ for every $\beta < \alpha$, then $u_{\alpha}= \sup\{u_{\beta} \mid \beta < \alpha \} \subset w$. We prove by induction that $u_{\alpha} \geq \alpha $. For $\alpha \in \{0,1,2\} $, it is obvious. For $\beta>1$, since $u_{\beta+1} \geq u_{\beta}^2 \geq \beta^2$, from Proposition~\ref{corpuri}, it follows that $u_{\beta+1} \geq \beta +1$. For a limit ordinal $\alpha$, we have that $u_{\alpha}= \sup\{u_{\beta} \mid \beta < \alpha \} \geq \sup\{\beta \mid \beta < \alpha\} = \alpha $. Thus, $w \geq \alpha$ for every ordinal $\alpha $, a contradiction.  
		So there exists an ordinal $y$ such that $u_y = w$. Then $y$ is the smallest ordinal $\alpha$ such that $ u_{\alpha} $ is quadratically closed. So for any $\alpha < y$, from Proposition~\ref{urm-corp-2}, we have that $ u_{\alpha+1} = u_{\alpha}^2 $. Then $w = u^{2^y}$. It remains to be proven that $y$ is the smallest $ \varepsilon$-number greater than $v$.
		
		Because $v \neq 0$, $L(u)$ is non-empty, and so is $u\setminus P[u]$, by Proposition~\ref{L-nevid}. Then $u$ is not quadratically closed, so $y >0$. We know, by Corollary~\ref{cor-L-nevid}, that $L(u_y) = \emptyset$ and for every $\alpha < y$, $L(u_{\alpha}) \neq \emptyset $. For an ordinal $\alpha < y$, we denote $\lambda_{\alpha} := \min L(u_{\alpha})$. From Lemma~\ref{lema-u-x-star} and Proposition~\ref{minim-L} it follows that $u_{\alpha}=\lambda_{\alpha}^*$, for every $\alpha < y$.
		
		For a limit ordinal $\delta \leq y$, we prove that:
		
		$$
		L(u_{\delta}) = \bigcup_{\alpha < \delta} \bigcap_{\delta>\beta \geq \alpha} L(u_{\beta});
		$$
		
		$$
		L(u_{\delta}) = \bigcap_{\alpha < \delta} \bigcup_{\delta>\beta \geq \alpha} L(u_{\beta}).
		$$

		Let $\lambda \in L(u_{\delta})$. Then $\lambda$ is a group, and $\{  \lambda + s \mid s < \lambda  \}\cap P[u_{\delta}] = \emptyset$. There exists $\alpha < \delta $ such that $ \lambda \in u_\alpha $, so $\lambda \in L(u_{\alpha})$. For every $ \delta > \beta \geq \alpha$, we have that $\lambda \in u_{\beta}$ and $\{  \lambda + s \mid s < \lambda  \}\cap P[u_{\beta}]  \subset     \{  \lambda + s \mid s < \lambda  \}\cap P[u_{\delta}] = \emptyset$, so $\lambda \in L(u_{\beta})$. Thus $L(u_{\delta}) \subset \bigcup_{\alpha < \delta} \bigcap_{\delta>\beta \geq \alpha} L(u_{\beta})$ and $ L(u_{\delta}) \subset  \bigcap_{\alpha < \delta} \bigcup_{\delta>\beta \geq \alpha} L(u_{\beta}) $.
		
		Let $\lambda \in \bigcup_{\alpha < \delta} \bigcap_{\delta>\beta \geq \alpha} L(u_{\beta}) $. Then there exists $\alpha < \delta $ such that $\lambda \in \bigcap_{\delta>\beta \geq \alpha} L(u_{\beta})$. It follows that $ \{  \lambda + s \mid s < \lambda  \}\cap P[u_{\beta}] = \emptyset$, for every $ \delta > \beta \geq \alpha$. Since $\bigcup_{\delta>\beta \geq \alpha} P[u_{\beta}] = P[u_{\delta}] $, it follows that $ \{  \lambda + s \mid s < \lambda  \}\cap P[u_{\delta}] = \emptyset$.
		Obviously $\lambda$ is a group, $\lambda \in u_{\delta}$, so $ \lambda \in L(u_{\delta}) $. 
		
		Let $ \lambda \in \bigcap_{\alpha < \delta} \bigcup_{\delta>\beta \geq \alpha} L(u_{\beta}) $. Then, for every $\alpha < \delta$, there exists $ \alpha \leq \beta_{\alpha} < \delta $ such that $ \lambda \in L(u_{\beta_{\alpha}}) $.
		So $\lambda \in u_{\delta}$ and is a group and $ \{  \lambda + s \mid s < \lambda  \}\cap P[u_{\beta_{\alpha}}] = \emptyset$, for every $\alpha < \delta$. Since $\bigcup_{\alpha < \delta} P[u_{\beta_{\alpha}}] = P[u_{\delta}] $, it follows that $ \{  \lambda + s \mid s < \lambda  \}\cap P[u_{\delta}] = \emptyset$. So $\lambda \in L(u_{\delta})  $.

		For an ordinal $\alpha < y$, we prove that:
		$$ L(u_{\alpha +1}) = (L(u_{\alpha}) \setminus \{ \lambda_{\alpha} \})  \cup \{  u_{\alpha} \times  \lambda  \mid \lambda \in L(u_{\alpha})  \}.  $$
		
		Let $ \lambda \in L(u_{\alpha+1}) $. Obviously $\lambda \neq \lambda_{\alpha}$, because $\lambda_{\alpha} = P(u_{\alpha})$. If $\lambda < u_{\alpha}$, since $P[u_{\alpha}] \subset P[u_{\alpha+1}]$, it follows that $\lambda \in L(u_{\alpha})$.
		
		If $ \lambda \geq u_{\alpha} $, since $\lambda \in u_{\alpha}(u_{\alpha})$, it follows that $\lambda = u_{\alpha} \times x + y$, where $x, y < u_{\alpha}$, and $x>0$.
		
		If $ x \notin L(u_{\alpha}) $, then there exist $\beta \in u_{\alpha}$, and $x_1, \dots, x_n < x$ such that $x \oplus P(\beta) = x_1 \oplus \dots \oplus x_n $. Let $\lambda_1 = u_{\alpha} \otimes x_1 \oplus y \oplus \beta \boxedexp{2} \otimes \lambda_{\alpha} $, $\lambda_2 =  u_{\alpha} \otimes x_2, \dots ,\lambda_n =  u_{\alpha} \otimes x_n $. Since $\lambda_{\alpha}, \beta < u_{\alpha}$ and $x_1, \dots , x_n < x$, it follows that $\lambda_1, \dots , \lambda_n < \lambda$, so $\lambda_1 \oplus \dots \oplus \lambda_n < \lambda$. We have that $\lambda \oplus (\lambda_1 \oplus \dots \oplus \lambda_n) = \lambda \oplus  u_{\alpha} \otimes (x \oplus P(\beta)) \oplus \beta \boxedexp{2} \otimes \lambda_{\alpha} = P(u_{\alpha} \otimes \beta)$, a contradiction.
		
		If $y>0$, then $\lambda = (u_{\alpha} \times x) \oplus y$ and $y, \ u_{\alpha} \times x< \lambda$, so $\lambda$ would not be a group, a contradiction.
		
		Thus $ L(u_{\alpha +1}) \subset (L(u_{\alpha}) \setminus \{ \lambda_{\alpha} \})  \cup \{  u_{\alpha} \times  \lambda  \mid \lambda \in L(u_{\alpha})  \}$.
		
		Let $\lambda \in L(u_{\alpha}) \setminus \{ \lambda_{\alpha} \} $. Then $\lambda$ is a group and $\lambda > \lambda_{\alpha}$. Assume towards a contradiction that there exist $ \lambda'<\lambda, \ \beta \in u_{\alpha+1} $ such that $\lambda \oplus  \lambda' = P(\beta)  $.
		There exist $x, y < u_{\alpha}$ such that $ \beta = u_{\alpha } \otimes x \oplus y $. Then $P(\beta)= u_{\alpha} \otimes (x\boxedexp{2} \oplus x) \oplus \lambda_{\alpha}\otimes x\boxedexp{2} \oplus y \boxedexp{2} \oplus y $.
		It follows that $ x\boxedexp{2} \oplus x = 0 $, so  $x = 0$ or $x =1 $, and also $ x\boxedexp{2} \otimes \lambda_{\alpha} \oplus y\boxedexp{2} \oplus y = \lambda \oplus \lambda' $. 
		
		If $x=0$, then $\lambda \oplus \lambda' = y\boxedexp{2} \oplus y$, contradicting $\lambda \in L(u_{\alpha})$. If $x=1$, then $ \lambda \oplus (\lambda' \oplus \lambda_{\alpha}) = y\boxedexp{2} \oplus y $, with $ \lambda', \lambda_{\alpha} < \lambda $, contradicting $\lambda \in L(u_{\alpha})$. Thus $ L(u_{\alpha}) \setminus \{ \lambda_{\alpha} \} \subset L(u_{\alpha+1}) $.
		
		Let $ \lambda = u_{\alpha} \times a $, where $a \in L(u_{\alpha})$. Since $u_{\alpha}, \ a $ are groups, it follows from Lemma~\ref{lema-produs-grupuri} that $\lambda$ is a group. Assume towards a contradiction that there exist $\lambda' < \lambda $, and $\beta \in u_{\alpha+1} $ such that $ \lambda \oplus \lambda' = P(\beta) $. There exist $x, y  < u_{\alpha}$ such that $\beta = u_{\alpha} \times x + y $ and $a' < a,\ b< u_{\alpha}$ such that $ \lambda' = u_{\alpha} \times a' + b $. From $ \lambda \oplus \lambda' = P(\beta) $ it follows that $ x\boxedexp{2} \oplus x = a \oplus a'  $. But $a \in L(u_{\alpha})$, a contradiction. So $  \{  u_{\alpha} \times  \lambda  \mid \lambda \in L(u_{\alpha})  \} \subset L(u_{\alpha+1})  $.
		
		For any ordinal $\alpha \leq y$, we define:
		
		$$ M(\alpha ) = \bigcup_{\beta \leq \alpha } L(u_{\beta}). $$
		
		We prove that for any $\alpha \leq  \alpha' \leq y$, $M(\alpha)$ is an initial segment of $M(\alpha')$.
		
		Clearly $M(\alpha) \subset M(\alpha') $.
		
		Let $\lambda \in M(\alpha')\setminus M(\alpha)$. Then $ \lambda \in L(u_{\beta}) $ for some $\beta > \alpha$. 
		
		We prove by induction on $\beta > \alpha$ that for any $\lambda \in L(u_{\beta}) \setminus M(\alpha) $, $\lambda > x$ for any $x \in M(\alpha)$.
		
		If $\beta = \alpha + 1$, the statement to be proven follows from the fact that $M(\alpha) \subset u_{\alpha} $ and $ L(u_{\alpha +1}) = (L(u_{\alpha}) \setminus \{ \lambda_{\alpha} \})  \cup \{  u_{\alpha} \times  \lambda  \mid \lambda \in L(u_{\alpha})  \}$.
		
		If $ \beta = \gamma +1 $, then $ L(u_{\gamma +1}) = (L(u_{\gamma}) \setminus \{ \lambda_{\gamma} \})  \cup \{  u_{\gamma} \times  \lambda  \mid \lambda \in L(u_{\gamma})  \}$ and the statement follows from the induction hypothesis and from the fact that any element from $\{  u_{\gamma} \times  \lambda  \mid \lambda \in L(u_{\gamma})  \}$ is at least $u_{\gamma}$, which is greater than or equal to $ u_{\alpha} $.
		
		If $\beta $ is a limit ordinal, then from the relation $L(u_{\beta}) = \bigcup_{a < \beta} \bigcap_{\beta> \delta \geq a} L(u_{\beta})$ it follows that there exists $ \alpha <  \gamma < \beta $ such that $\lambda \in L(u_\gamma) $ and we can apply the induction hypothesis. 
		
		Thus $\lambda > x$, for every $x \in M(\alpha)$. Since $\lambda \in M(\alpha')\setminus M(\alpha) $ was arbitrary, it follows that $M(\alpha)$ is an initial segment of $M(\alpha')$.
		
		We notice that, for all $\alpha \leq y$ limit ordinal, $M(\alpha) = \bigcup_{\delta < \alpha} M(\delta) $.
		
		We observe that from the relation $L(u_{\alpha +1}) = (L(u_{\alpha}) \setminus \{ \lambda_{\alpha} \})  \cup \{  u_{\alpha} \times  \lambda  \mid \lambda \in L(u_{\alpha})  \}$ it follows that for an ordinal $\alpha < y$, $L(u_{\alpha+1}) \setminus L(u_{\alpha}) $ is non-empty, because $L(u_{\alpha})$ is non-empty. Thus, it follows by induction that, for all $\alpha \leq y$, $\otype(M(\alpha)) \geq \alpha $.

		For an ordinal $\alpha \leq y$, we can define:
		$$ T(\alpha) = \{ en_{M(\alpha)}(x)   \mid x<\alpha  \}. $$
		
		We prove that if $\alpha \leq y $ is a limit ordinal, then:
		$$  \bigcup_{\beta < \alpha} T(\beta) = T(\alpha). $$
		
		Let $\lambda = en_{M(\alpha)}(x)$, with $x<\alpha$. There exists $\beta<\alpha$ such that $x<\beta$. Since $M(\beta)$ is an initial segment of $M(\alpha)$, it follows that $\lambda \in T(\beta)$. So $T(\alpha) \subset  \bigcup_{\beta < \alpha} T(\beta) $. Let $\lambda \in T(\beta)$, for some $\beta < \alpha$. We denote $\lambda = en_{M(\beta)}(x)$, where $x < \beta$. Then, since $M(\beta)$ is an initial segment of $M(\alpha)$, it follows that $\lambda = en_{M(\alpha)}(x) \in T(\alpha) $. So the relation is proven.
		
		We observe that from the fact that for any $\alpha < \alpha' \leq y$, $M(\alpha)$ is an initial segment of $M(\alpha')$, it follows that $T(\alpha)$ is an initial segment of $T(\alpha')$. Thus, for all $\alpha \leq y$ limit ordinal and for all $\delta < \alpha$:
		$$   \bigcup_{\delta<\beta < \alpha} T(\beta) = T(\alpha)  $$
		
		We prove by induction on $\alpha \leq y$ that $L(u_{\alpha}) = M(\alpha) \setminus T(\alpha)$.
		
		For $\alpha = 0$, $L(u_0) = M(0)$.
		
		For $\alpha = \beta +1$,  $M(\beta + 1) = M(\beta) \cup L(u_{\beta+1}) = L(u_{\beta}) \cup T(\beta) \cup \{  u_{\beta} \times  \lambda  \mid \lambda \in L(u_{\beta})  \} = L(u_{\beta+1}) \cup T(\beta) \cup \{ \lambda_{\beta} \}$. But $\lambda_{\beta} = \min L(u_{\beta}) =  \min (M(\beta) \setminus T(\beta))$ and $M(\beta)$ is an initial segment of $M(\beta+1)$, so $T(\beta+1) = T(\beta ) \cup \{ \lambda_{\beta} \} $. Since $ L(u_{\beta +1}) = (L(u_{\beta}) \setminus \{ \lambda_{\beta} \})  \cup \{  u_{\beta} \times  \lambda  \mid \lambda \in L(u_{\beta})  \}$ and $L(u_{\beta}) \cap T(\beta) = \emptyset$, it follows that $L(u_{\beta+1}) \cap T(\beta+1) = \emptyset$, so $L(u_{\alpha}) = M(\alpha) \setminus T(\alpha)$. 
		
		For a limit ordinal $\alpha$, $ L(u_{\alpha}) = \bigcup_{\alpha > \delta} \bigcap_{\alpha>\beta \geq \delta} L(u_{\beta})  = \bigcup_{\alpha > \delta} \bigcap_{\alpha>\beta \geq \delta}(M(\beta) \setminus T(\beta) ) $. 
		
		Let $ \lambda \in \bigcup_{\alpha > \delta} \bigcap_{\alpha>\beta \geq \delta}(M(\beta) \setminus T(\beta) ) $, then there exists $\delta < \alpha$ such that $ \lambda \in M(\beta) \setminus T(\beta) $, for every $ \beta \geq \delta $. It follows that $ \lambda \in M(\alpha) $ and $\lambda \notin T(\beta)$, for every $\alpha > \beta  \geq \delta$. So $\lambda \notin T(\alpha)$, from which it follows that $ \lambda \in M(\alpha) \setminus T(\alpha) $.
		
		Let $\lambda \in M(\alpha) \setminus T(\alpha) $. Since $M(\alpha) = \bigcup_{\delta < \alpha} M(\delta) $,  there exists $\delta < \alpha$ such that $\lambda \in M(\delta)$. So $\lambda \in M(\beta)$, for every $\alpha > \beta \geq \delta$. Since $\lambda \notin T(\alpha)$, it follows that $\lambda \notin T(\beta)$, for every $\alpha > \beta \geq \delta$. So $\lambda \in \bigcap_{\alpha>\beta \geq \delta } L(u_{\beta}) \subset L(u_{\alpha})$.
		
		We define the function $f: y^+  \to \ordinal $ such that $f(\alpha) = \otype(M(\alpha)) $, for every $ \alpha \leq y $. 
		We define $ g: y^+  \to \ordinal  $ such that $g(\alpha) = \otype(L(u_{\alpha}))$, for every $ \alpha \leq y $. 
		
		From the relation $L(u_{\alpha}) = M(\alpha) \setminus T(\alpha)$ and from the definition of $T(\alpha)$ it follows that for every $ \alpha \leq y $, $ f(\alpha) = \alpha +g(\alpha) $. Since $g(\alpha) > 0$, it follows that $f(\alpha) > \alpha$, for every $\alpha < y$. Since $L(u_y) = \emptyset$, it follows that $g(y) = 0$, so $f(y)=y$.
		
		From the hypothesis we have that:
		$$ f(0) = \otype(L(u)) = v.$$
		
		For all $\alpha<y$, since $ M(\alpha + 1) = M(\alpha) \cup L(u_{\alpha+1}) = M(\alpha)  \cup \{  u_{\alpha} \times  \lambda  \mid \lambda \in L(u_{\alpha})  \} $ and every element from $ \{  u_{\alpha} \times  \lambda  \mid \lambda \in L(u_{\alpha}) \} $ is greater than $u_{\alpha}$, which is strictly greater than any element in $ M(\alpha) $, it follows that:
		$$  f(\alpha+1) = f(\alpha) + g(\alpha), \text{ for every } \alpha < y . $$
		
		Let $\alpha \leq y $ be a limit ordinal. Since $M(\alpha) = \bigcup_{\beta<\alpha} M(\beta) $ and $M(\beta)$ is an initial segment of $M(\alpha)$, for every $\beta < \alpha$, it follows that:
		$$ f(\alpha) = \sup \{ f(\beta) \mid \beta < \alpha \}, \text{ for every limit ordinal } \alpha.  $$

		We treat two cases, depending on the size of the ordinal $v$. 
		
		\underline{Case 1:} If $1 \leq v < \omega$.
		
		From the recurrence relation for $f(\alpha+1)$, it follows that $f(n+1) = 2f(n)-n$, for every $n < \omega$. It follows inductively that $ f(n) = 2^n (v-1) +n +1 > n$, so $y>n$, for every $n \in \omega$. But $f(\omega )=\sup\{f(n) \mid n \in \omega \} = \omega$, so $y = \omega$. Thus $y$ is the smallest $\varepsilon $-number greater than $v$. 
		
		\underline{Case 2:} If $v \geq \omega $. 
		
		Let $\varepsilon$ be the smallest $\varepsilon $-number greater than $v$. 
		
		We prove by induction on $n < \omega$ that $f(n) = v \times 2^n$. We know that $f(0) = v$ and $f(n+1) = f(n) + g(n) = f(n) \times 2 = v \times 2^{n+1} $, because $f(n) = g(n) $ by Lemma~\ref{1+v=v}.
		
		It follows that $f(\omega) = v \times 2^{\omega} = v \times \omega \geq \omega \times \omega > \omega $, so $y > \omega$.
		 
		We prove that $\varepsilon \geq y $.

		We prove that $f(\alpha) \leq v \times 2^{\alpha} $, for every $\alpha \leq y$. If $\alpha < \omega$, we have proved that equality holds. If $\alpha = \beta + 1$, then $f(\alpha) = f(\beta) + g(\beta) \leq f(\beta) \times 2 \leq v \times 2^{\beta+1}$. If $\alpha$ is a limit ordinal, then $f(\alpha) = \sup\{ f(\beta) \mid \beta < \alpha \} \leq \sup \{ v \times 2^{\beta}  \mid \beta < \alpha  \} = v \times 2^{\alpha} $.
		
		Assume that $\varepsilon < y$. Then $f(\varepsilon) \leq v \times 2^{\varepsilon} =v \times \varepsilon =\varepsilon $, by Lemma~\ref{lema-epsilon}. We obtain a contradiction with $ f(\varepsilon) > \varepsilon $. Thus $\varepsilon \geq y $. 
		
		We prove by induction on $\beta$ that if $ \omega \leq \beta < \varepsilon  $, then $ \beta \leq y $ and $ f(\beta) \geq \beta + 2^{\beta} $.

		If $\beta = \omega$, then we know that $ y > \omega $ and $ f(\omega) = v \times \omega \geq \omega \times \omega > \omega \times 2 = \omega + 2^{\omega} $.
		
		If $\beta = \beta'+1$, then, by induction, $f(\beta') \geq \beta' + 2^{\beta'} > \beta' $, from which it follows that $ \beta'<y $, so $\beta \leq y$. We know that $f(\beta) = f(\beta') + g(\beta') \geq \beta' + 2^{\beta'} + 2^{\beta'}  = \beta' + 2^{\beta} = \beta' + (1+2^{\beta}) = \beta + 2^{\beta} $, by Lemma~\ref{1+v=v}, because $2^{\beta} \geq \omega$.
		
		If $ \beta $ is a limit ordinal, we know that $\delta \leq y$, for every $\delta < \beta$, so $ \beta \leq y $. We have that $ f(\beta ) = \sup \{ f(\delta) \mid \delta < \beta  \} \geq \sup \{ \delta + 2^{\delta} \mid \delta < \beta \} \geq 2^{\beta}  $. If $ \beta < 2^{\beta}$, then, from Lemma~\ref{ineg-ord-pt-claim}, it follows that $2^{\beta} = \beta + 2^{\beta}$. If $\beta = 2^{\beta} $, then $\beta$ is an $\varepsilon$-number strictly smaller than $\varepsilon$, so $\beta \leq v$. From the recurrence relations for $f$, it follows that $f$ is strictly increasing. Then $f(\alpha) \geq v + \alpha $, for every $ \alpha \leq y $. Thus $ f(\beta) \geq v + \beta \geq \beta + \beta = \beta +2^{\beta} $.
		
		Then, for every $\beta < \varepsilon$, we have $\beta \leq y$, and since $\varepsilon $ is a limit ordinal from Remark~\ref{obs-epsilon-nr}, it follows that $\varepsilon \leq y$.
		
		Thus $y = \varepsilon$ and the proof is complete.
	\end{proof}

	\subsection{The quadratic closure of the next field after a transcendental} \label{sec: 3.4}
	
	The results in this subsection, except for the proof of Proposition~\ref{epsilon^3}, were published by H. Lenstra in \cite{nim-mult}.
	
	In order to apply Theorem~\ref{claim-inch-patr} for the perfect closure of $\tau(\tau)$, we need to find $\otype (L(\tau^{\tau \times \omega}))$. We will find $L(\tau(\tau))$ and then prove that $L(\tau(\tau)) = L(\tau^{\tau \times \omega})$.
	
	\begin{proposition} \label{L-de-urm-corp}
		Let $t$ be an algebraically closed field ordinal. We denote $u = t(t) = t^t$ and $t=2^s$. Then:
		\begin{align*}
			L(u) = \{ t\boxedexp{2n+1} \otimes \lambda, \ \lambda \oslash (t \oplus \alpha)\boxedexp{2n+1} \mid  n \in \omega, \ \alpha, \lambda \in t, \ \lambda \text{ is a group}  \}.
		\end{align*}
		
		Moreover, $\otype(L(u))=s\times t$.
		
	\end{proposition}
	
	\begin{proof}
		We denote $A = \{ t\boxedexp{2n+1} \otimes \lambda, \ \lambda \oslash (t \oplus \alpha)\boxedexp{2n+1} \mid  n \in \omega, \ \alpha, \lambda \in t, \ \lambda \text{ is a group}  \}$. From the identities proved in Proposition~\ref{urm-corp-1}, it follows that $A = \{ t^{2n+1} \times \lambda, \  t^{\omega + \omega \times \alpha +2n}  \times \lambda \mid  n \in \omega, \ \alpha, \lambda \in t, \ \lambda \text{ is a group}  \}$.
		
		Let $\lambda \in L(u)$. From Lemma~\ref{def-alternativa-Lu}, it follows that $\lambda$ is a group, so $\lambda = 2^{\alpha}$, where $\alpha< s \times t$, $\alpha = s \times x + y$, with $x<t$, $y<s$. Then $\lambda = t^x \times 2^y$. Since $t = \omega + \omega \times t$, then $x=n<\omega$ or $x = \omega + \omega \times \beta +n$, where $\beta <t$, $n<\omega$. We see that $x$ cannot be zero, since, if that were so, then $\lambda = 2^y < t$, so $\lambda \in L(t) = \emptyset$, a contradiction.
		
		If $x=n<\omega$, assume that $n$ is even, so there is an $m \in \mathbb{N}^*$ with $n=2m$. Since $2^y <t$, there exists $\xi < t $ such that $2^y = \xi\boxedexp{2}$. 
		Then $P(t\boxedexp{m} \otimes \xi) = \lambda \oplus (t\boxedexp{m} \otimes \xi )$. But $t\boxedexp{m} \otimes \xi = t^m \times \xi < t^m \times t \leq t^n \leq \lambda$, a contradiction. Thus $n$ is odd and $\lambda \in A$.
		
		If $x = \omega + \omega \times \beta +n$, where $\beta <t$, $n<\omega$, assume that $n$ is odd, $n=2m-1$. From the  identities proved in Proposition~\ref{urm-corp-1}, it follows that $\lambda = 2^y \oslash (t \oplus \beta )\boxedexp{2m} $. Since $2^y <t$, there exists $\xi < t $ such that $2^y = \xi\boxedexp{2}$. Then $P(\xi \oslash (t \oplus \beta )\boxedexp{m} ) = \lambda \oplus (\xi \oslash (t \oplus \beta )\boxedexp{m} )$ and $\xi \oslash (t \oplus \beta )\boxedexp{m} = t^{\omega +\omega \times \beta + m-1} \times \xi < \lambda$, a contradiction. Thus $n$ is even and $\lambda \in A$.
		
		Hence $L(u) \subset A $.
		
		Let $n \in \omega , \ \lambda \in t $ be a group. From Lemma~\ref{lema-produs-grupuri} it follows that $ \eta = t^{2n+1} \times \lambda$ is a group. Let $ \lambda_1 < t^{2n+1} \times \lambda $, $\lambda_1 = t^{2n+1} \times z + v$, where $z<\lambda, \ v<t^{2n+1}$. Then, there exists $ S \in t[X] $ of degree strictly less than $2n+1$ with $v = S(t)$. Assume towards a contradiction that $ \eta \oplus \lambda_1 =P(\mu) $, for some $ \mu \in u $, $\mu = Q(t) \oslash R(t)$, where $Q, R \in t[X], \ R \neq 0$. We know that $ \eta \oplus \lambda_1  =  t\boxedexp{2n+1}\otimes(\lambda \oplus z) \oplus v $, and we denote $p(X) = X\boxedexp{2n+1}\otimes(\lambda \oplus z) \oplus S$. Clearly $p \in t[X], \ \deg(p) = 2n+1$. It follows that $ p(X) \otimes R(X)\boxedexp{2} =Q(X)\boxedexp{2} \oplus Q(X)\otimes R(X)$. We look at the degrees of the polynomials in this relation. If $\deg(Q)>\deg(R)$, then $ \deg(p) + 2\deg(R) = 2\deg(Q)$, a contradiction. If $\deg(Q) \leq \deg(R)$, then $ \deg(p) + 2\deg(R)\leq 2 \deg(R) $, a contradiction. Thus $t^{2n+1} \times \lambda \in L(u) $.

		Let $n \in \omega, \ \alpha, \lambda \in t$, where $\lambda$ is a group. From Lemma~\ref{lema-produs-grupuri} it follows that $ \eta =  t^{\omega + \omega \times \alpha +2n}  \times \lambda $ is a group. Let $ \lambda_1 < t^{\omega + \omega \times \alpha +2n} \times \lambda $, $\lambda_1 = t^{\omega + \omega \times \alpha +2n} \times z + v$, where $z<\lambda, \ v<t^{\omega + \omega \times \alpha +2n} $. Assume towards a contradiction that $ \eta \oplus \lambda_1 =P(\mu) $, for some $ \mu \in u $, $\mu = Q(t) \oslash R(t)$, where $Q, R \in t[X], \ R \neq 0$, and $Q, \ R$ are coprime. We know that $ \eta \oplus \lambda_1  =  (\lambda \oplus z) \oslash (t \oplus \alpha )\boxedexp{2n+1}  \oplus v $, and we denote $x = \lambda \oplus z \neq 0$. From the characterization of $t^{\omega + \omega \times \alpha +2n}$ in Proposition~\ref{urm-corp-1}, we have that $ x \oplus v \otimes (t \oplus \alpha )\boxedexp{2n+1} = p(t) \oslash (\bigotimes_{i=1}^r(t \oplus \beta_i)\boxedexp{m_i}) $, where $p \in t[X]$, $p(\alpha) \neq 0$, $r \in \omega, \ \beta_i < \alpha, \ m_i \in \omega$, for every $ i \in \{1, \dots , r\}$. Then $p(X) \otimes R(X)\boxedexp{2} = (Q(X)\boxedexp{2} \oplus Q(X)\otimes R(X))\otimes (X \oplus \alpha )\boxedexp{2n+1} \otimes (\bigotimes_{i=1}^r(X \oplus \beta_i)\boxedexp{m_i}) $. It follows that $(X \oplus \alpha)\boxedexp{n+1} \mid R$, so $(X \oplus \alpha ) \mid Q\boxedexp{2} $, which contradicts the fact that $Q$ and $R$ are coprime. Thus $t^{\omega + \omega \times \alpha +2n} \times \lambda \in L(u)$.
		
		Hence $A = L(u)$.
		
		We define $\Phi : s \times t \to L(u)$ such that for every $n \in \omega , \ r<s, \ \beta<t $:
		$$ \Phi(s \times n + r) = t^{2n+1} \times 2^r; $$
		$$ \Phi (s \times (\omega + \omega \times \beta +n) + r) = t^{\omega + \omega \times \beta +2n} \times 2^r. $$
		
		Then $\Phi$ is well-defined, bijective, and order-preserving, so $\otype(L(u))=s\times t$.
	\end{proof}
	
	\begin{proposition} \label{L-pur-insep}
		Let $u' \hookrightarrow u $ be an algebraic and purely inseparable field extension. Then $L(u)=L(u')$.
	\end{proposition}
	
	\begin{proof}
		We prove that $L(u') \subset L(u)$. Let $\lambda \in L(u')$. From Lemma~\ref{def-alternativa-Lu}, it follows that $\lambda$ is a group. Assume towards a contradiction that there exist $x \in  u, \ s \in \lambda $ such that $x\boxedexp{2}  \oplus x = \lambda \oplus s $. Then the minimal polynomial of $x$ over $u'$ divides the polynomial $X\boxedexp{2} \oplus X \oplus \lambda \oplus s $. From Proposition~\ref{pur-insep}, it follows that the minimal polynomial of $x$ over $u'$ cannot be the polynomial $X\boxedexp{2} \oplus X \oplus \lambda \oplus s $, therefore $x \in u'$, a contradiction with $ \lambda \in L(u')$.
		
		We prove that $L(u) \subset L(u')$. Let $\lambda \in L(u)$. From Lemma~\ref{def-alternativa-Lu}, it follows that $\lambda$ is a group. Assume towards a contradiction that $\lambda \notin L(u')$. Then $\lambda \geq u'$. From Proposition~\ref{pur-insep}, it follows that there exists $n \in \omega $ such that $ \lambda \boxedexp{2^n} \in u'$. Then $ \lambda \oplus \lambda \boxedexp{2^n} = (\lambda \oplus \lambda \boxedexp{2}) \oplus \dots \oplus (\lambda \boxedexp{2^{n-1}} \oplus \lambda \boxedexp{2^n})$. So $\lambda \oplus \lambda \boxedexp{2^n} \in P[u]$, and $\lambda \boxedexp{2^n} < u' < \lambda$, a contradiction. 
	\end{proof}
	
	\begin{remark} \label{obs-egal-L}
		From Proposition~\ref{inch-perfecta-ord-corp-alg-inchis}, Proposition~\ref{corolar-pur-insep}, and Proposition~\ref{L-pur-insep} it follows that for a transcendental ordinal $t$, we have that $L(t^t)=L(t^{t \times \omega })$.
	\end{remark}
	
	\begin{theorem} \label{inch-patratica-transcendent}
		Let $t$ be a transcendental ordinal. We define the sequence $(a_{n})_{n \in \omega} $ such that $ a_0 =t $, $a_{n+1}=t^{a_n}$, for every $n \in \omega$. Then the quadratic closure of $t(t)$ is $\sup \{ a_n \mid n \in \omega \}$.
	\end{theorem}
	
	\begin{proof}
		From Proposition~\ref{inch-perfecta-ord-corp-alg-inchis} it follows that $t^{t \times \omega}$ is the perfect closure of $t(t)$, so the quadratic closure of $t(t)$ is the quadratic closure of $t^{t \times \omega}$. From Remark~\ref{obs-egal-L} and Proposition~\ref{L-de-urm-corp}, we have that $ \otype(L(t^{t\times \omega}))=s \times t $, where $2^s =t$. From Theorem~\ref{claim-inch-patr}, it follows that the quadratic closure of $t^{t \times \omega}$ is $ (t^{t \times \omega}) ^{\varepsilon} $, where $\varepsilon $ is the smallest $\varepsilon$-number strictly greater than $s\times t$.
		
		Since $t>\omega$, it follows that $\varepsilon > \omega$. From Lemma~\ref{lema-epsilon} it follows that $ \omega \times \varepsilon = t \times \varepsilon = \varepsilon $, so $t^{t \times \omega \times \varepsilon} = t^{\varepsilon}$. 
		
		We observe that $a_{n+1} \leq s \times a_{n+1}= s \times t^{a_n} \leq t \times t^{a_n} = t^{1+a_n} = t^{a_n} = a_{n+1}  $, from Lemma~\ref{1+v=v}. So $a_{n+1} = s \times a_{n+1}$, for every $n \in \omega$.
		
		We prove that $ \varepsilon \geq  s \times  a_n $, for every $n \in \omega$. Clearly $ \varepsilon \geq s \times a_0 $.
		
		If $ \varepsilon \geq s \times a_n $, it follows that $ \varepsilon = 2^{\varepsilon} \geq 2^{s \times a_n} = a_{n+1} $. So $ \varepsilon \geq s \times a_{n+1} $. 
		
		Then $ \varepsilon \geq \sup \{ a_n \mid n \in \omega \} $. But $2^{\sup \{ a_n \mid n \in \omega \}} = \sup \{ 2^{a_{n+1}} \mid n \in \omega \} = \sup \{ 2^{s \times a_{n+1}} \mid n \in \omega \} = \sup \{ a_{n+2} \mid n \in \omega \} = {\sup \{ a_n \mid n \in \omega \}} $, because the sequence $(a_n)_n$ is increasing. Thus $ \varepsilon = \sup \{ a_n \mid n \in \omega \} $ and $ t^{\varepsilon} = \sup \{ a_n \mid n \in \omega \} $, which was to be proven.
	\end{proof}
	
	\begin{remark} \label{obs-tau-inch-patr}
			From Theorem~\ref{inch-patratica-transcendent}, Proposition~\ref{epsilon-0-tau}, and Proposition~\ref{epsilon-0} it follows that the quadratic closure of $ \tau(\tau) $ is $\varepsilon_0$.
	\end{remark}

	\begin{proposition} \label{epsilon^3}
		Let $t$ be a transcendental ordinal. Let $ \theta $ be the quadratic closure of $t(t)$. Then $\theta \boxedexp{3} = t $.
	\end{proposition}

	\begin{proof}
		From Theorem~\ref{set} g), it is enough to prove that the smallest irreducible polynomial over $ \theta $ is $ X\boxedexp{3} \oplus t $. Since $\theta $ is quadratically closed and $t$ is algebraically closed, it follows that any polynomial smaller than $ X\boxedexp{3} \oplus t $ is reducible. Being of degree 3, in order to prove that $ X\boxedexp{3} \oplus t $ is irreducible it is enough to prove that it has no roots in $ \theta $.
		
		Assume towards a contradiction that there exists $x \in \theta $ such that $x\boxedexp{3} = t$.
		
		From Proposition~\ref{inch-perfecta-ord-corp-alg-inchis} it follows that $t^{t \times \omega }$ is the perfect closure of $t(t)$. Let $\varepsilon$ be the smallest $\varepsilon$-number strictly greater than $\otype(L(t^{t \times \omega}))$. Consider the sequence of fields as in  Theorem~\ref{claim-inch-patr}: $(u_{\alpha})_{\alpha \leq \varepsilon }$ such that $u_0 = t^{t \times \omega}$, $ u_{\alpha + 1} = u_{\alpha}(u_{\alpha}) $, and $ u_{\beta} = \sup \{ u_{\delta} \mid \delta < \beta \} $ if $\beta$ is a limit ordinal. Then it follows that $u_{\alpha + 1} = u_{\alpha}^2$ for every $\alpha < \varepsilon $. By induction it follows immediately that $ u_{\alpha} = u_0^{ 2^{\alpha} } $ for every $\alpha \leq \varepsilon $. Thus $u_{\varepsilon} = \theta$.
		
		We prove by induction that $ x \notin u_{\alpha} $ for every $\alpha \leq \varepsilon $.
		
		To prove that $x \notin u_0$, we prove that $x \notin t^{t \times 2^n} $ for every $n \in \omega$.
		
		If $x \in t(t) $, then $x = Q(t) \oslash R(t)$, where $Q, R \in t[X]$ are coprime and $R \neq 0 $. Then $ R(t)\boxedexp{3}\otimes t = Q(t)\boxedexp{3} $, so $ R(X)\boxedexp{3}\otimes X = Q(X)\boxedexp{3}  $, a contradiction because the degree of the polynomial on the left-hand side is not divisible by 3, while that of the polynomial on the right-hand side is. 
		
		If for some $n \in \omega $, $x \in t^{t \times 2^{n+1}} $, but $x \notin t^{t \times 2^{n}} $, then since the roots of $ X\boxedexp{3} \oplus t $ are $x$, $2 \otimes x$, $3 \otimes x$ and $ x \notin t^{t \times 2^{n}} $ it follows that $ X\boxedexp{3} \oplus t $ has no roots in $ t^{t \times 2^{n}} $ and being of degree 3, it follows that it is irreducible over $ t^{t \times 2^{n}} $. Then $[t^{t \times 2^{n}}(x): t^{t \times 2^{n}}] = 3$. But $t^{t \times 2^{n}} \subset t^{t \times 2^{n}}(x) \subset t^{t \times 2^{n+1}} $ and $[t^{t \times 2^{n+1}} : t^{t \times 2^{n}}] =2$ from Proposition~\ref{inch-perfecta-ord-corp-alg-inchis}, a contradiction.

		If $ \alpha = \beta +1 $, assume towards a contradiction that $ x \in u_{\alpha} $. Since the roots of $ X\boxedexp{3} \oplus t $ are $x$, $2 \otimes x$, $3 \otimes x$ and $ x \notin u_{\beta} $ it follows that $ X\boxedexp{3} \oplus t $ has no roots in $ u_{\beta} $ and being of degree 3, it follows that it is irreducible over $ u_{\beta} $. Then $[u_{\beta}(x): u_{\beta}] = 3$. But $u_{\beta} \subset u_{\beta}(x) \subset u_{\alpha} $ and $[u_{\alpha} : u_{\beta}] =2$, a contradiction.
		
		If $\alpha$ is a limit ordinal, since $u_{\alpha} = \sup \{  u_{\beta} \mid \beta < \alpha  \} $ and $ x \notin u_{\beta} $ for every $\beta < \alpha$ it follows that $ x \notin u_{\alpha} $.
		
		Hence $x \notin \theta $, a contradiction.
	\end{proof}
	
	\begin{remark} \label{obs-tau-radical-de-ordin-3}
		From Proposition~\ref{epsilon^3}, it follows that $\varepsilon_0\boxedexp{3}= \tau$.
	\end{remark}

	\section{New quadratically closed field ordinals} \label{sec: 4}
	
	We change H. Lenstra's definition for $P$ and generalize his definition for $L$. We will then use the new $P$ and $L$ to prove that $\{\varepsilon_\alpha \mid \alpha \leq \omega^{\omega^\omega} \}$ are the next quadratically closed field ordinals.

	\begin{definition}
		For any $ \alpha < \tau $, we define $ P_{\alpha}(X) = (\tau \oplus \alpha) \otimes X \boxedexp{4} \oplus X $. For an ordinal $u$, we denote $ P_{\alpha} [u] = \{ P_{\alpha}(x) \mid x \in u \} $.
	\end{definition}
	
	\begin{remark}
		For any $ \alpha < \tau $ and any ordinals $a$ and $b$, we have that $P_{\alpha}(a) \oplus P_{\alpha}(b) = P_{\alpha}(a \oplus b)$.
	\end{remark}

	\begin{definition}
		For any $ \alpha < \tau $ and any nonzero ordinal $u$, we define:
		\begin{align*}
			L_{\alpha}(u) &= \{ \lambda \in u \setminus \{0\} \mid \text{ there is no } \beta \in u \text{ such that } \lambda \oplus P_{\alpha}(\beta) \\ &  \text{ can be written as a finite Nim-sum of ordinals smaller than } \lambda \}.
		\end{align*}
	\end{definition}
	
	Some properties of $L$ remain the same for $L_{\alpha}$, such as the next three propositions.
	
	\begin{proposition}
		Let $ \alpha < \tau $ and $u$ be an ordinal. Then:
		\begin{align*}
			L_{\alpha}(u) = \{ \lambda \in u\setminus \{0\}  \mid \lambda \text{ is a group and there are no } \beta \in u, \ \lambda_1 < \lambda \text{ such that }  P_{\alpha}(\beta) = \lambda \oplus \lambda_1 \}.
		\end{align*} 
	\end{proposition}
	
	\begin{proof}
		Let $A=  \{ \lambda \in u\setminus \{0\}  \mid \lambda \text{ is a group and there are no } \beta \in u, \ \lambda_1 < \lambda \text{ such that }  P_{\alpha}(\beta) = \lambda \oplus \lambda_1 \} $.
		It is obvious that $ A \subset L_{\alpha}(u) $. 
		
		Let $\lambda \in L_{\alpha}(u)$. If $\lambda $ is not a group, then from Theorem~\ref{set} a), there exist $a,b < \lambda$ such that $a \oplus b = \lambda $. For $\beta = 0$, we have that $\lambda \oplus P_{\alpha}(\beta) = a \oplus b$, a contradiction. So $\lambda$ is a group and then $\lambda \in A  $ follows easily.
	\end{proof}

	\begin{proposition} \label{prop-L-general}
		Let $ \alpha < \tau $. Let $u$ and $u'$ be field ordinals such that $\tau(\tau) \hookrightarrow u' \hookrightarrow u$ are algebraic field extensions and $u' \hookrightarrow u$ is purely inseparable. Then $L_{\alpha}(u') = L_{\alpha}(u)$.
	\end{proposition}
	
	\begin{proof}
		Let $\lambda \in L_{\alpha}(u')$. Then $\lambda <u'$ is a group. Assume towards a contradiction that there exist $x \in u$, $s < \lambda$ such that $P_{\alpha}(x) = \lambda \oplus s$. Then the minimal polynomial of $x$ over the field $u'$ divides the polynomial $ (\tau \oplus \alpha) \otimes X \boxedexp{4} \oplus X \oplus (\lambda \oplus s) $. But, from Proposition~\ref{pur-insep}, the minimal polynomial of $x$ over the field $u'$ has the form $(X \oplus x ) \boxedexp{2^n}$, for some $n \in \omega$. Since $ ((\tau \oplus \alpha) \otimes X \boxedexp{4} \oplus X \oplus (\lambda \oplus s))' = 1 $, it follows that $(\tau \oplus \alpha ) \otimes X \boxedexp{4} \oplus X \oplus (\lambda \oplus s)$ has no double roots, so $n=0$, from which it follows that $x \in u'$, a contradiction with $\lambda \in L_{\alpha}(u')$. Thus $ L_{\alpha}(u') \subset  L_{\alpha}(u)$.
		
		Let $\lambda \in  L_{\alpha}(u)$. Assume towards a  contradiction that $\lambda \notin  L_{\alpha}(u')$. Then $\lambda \geq u'$. From Proposition~\ref{pur-insep} it follows that there exists $n \in \omega $ such that $\lambda \boxedexp{2^n} \in u'$, hence also $\lambda \boxedexp{4^n} \in u'$. We consider the sequence $(a_n)_n$ such that $a_1 = 1$ and $a_{n+1} = 4 a_n +1$, for any $n>0$. Then $(\tau \oplus \alpha) \boxedexp{a_n} \otimes \lambda \boxedexp{4^n} \oplus \lambda = ((\tau \oplus \alpha) \boxedexp{a_n} \otimes \lambda \boxedexp{4^n} \oplus (\tau \oplus \alpha) \boxedexp{a_{n-1}} \otimes \lambda \boxedexp{4^{n-1}}) \oplus ((\tau \oplus \alpha) \boxedexp{a_{n-1}} \otimes \lambda \boxedexp{4^{n-1}} \oplus (\tau \oplus \alpha) \boxedexp{a_{n-2}} \otimes \lambda \boxedexp{4^{n-2}}) \oplus \dots \oplus ((\tau \oplus \alpha) \otimes \lambda \boxedexp{4} \oplus \lambda) = P_{\alpha}((\tau \oplus \alpha) \boxedexp{a_{n-1}} \otimes \lambda \boxedexp{4^{n-1}} \oplus \dots \oplus \lambda  ) \in P_{\alpha}(u')$ and $(\tau \oplus \alpha) \boxedexp{a_{n}} \otimes \lambda \boxedexp{4^{n}}<u'\leq \lambda $, a contradiction with $\lambda \in L_{\alpha}(u)$. Thus $ L_{\alpha}(u) \subset  L_{\alpha}(u')$.
	\end{proof}
	
	\begin{remark} \label{obs-l-general-tau}
		From Proposition~\ref{inch-perfecta-ord-corp-alg-inchis}, Proposition~\ref{corolar-pur-insep} and Proposition~\ref{prop-L-general} it follows that $L_{\alpha}(\tau(\tau)) = L_{\alpha}(\tau ^{\tau \times \omega})$, for any $\alpha<\tau$.
	\end{remark}
	
		\begin{proposition} \label{minim-L-general}
		Let $\alpha < \tau$. Let $u$ be a field ordinal such that $ P_{\alpha}[u] \neq u $. Let $ x = \min(u \setminus P_{\alpha}[u])$. Then $x = \min L_{\alpha}(u)$. 
		
		In particular, if $ P_{\alpha}[u] \neq u $, then $L_{\alpha}(u)$ is non-empty. 
	\end{proposition}
	
	\begin{proof}
		We prove that $x \in L_{\alpha}(u)$. Suppose by contradiction that $x \notin L_{\alpha}(u) $, then there exist $\beta \in u$, $ \lambda_1, \dots , \lambda_n < x $ such that $x \oplus P_{\alpha}(\beta) = \lambda_1 \oplus \dots \oplus \lambda_n$. Since $ x = \min(u \setminus P_{\alpha}[u])$, it follows that 
		$ \lambda_1, \dots , \lambda_n \in P_{\alpha}[u]$. Since $P_{\alpha}[u]$ is an additive subgroup of $u$, it follows that $ x \in P_{\alpha}[u] $, a contradiction.
		
		Let $y \in L_{\alpha}(u)$. Suppose by contradiction that $x>y$. Then $y \in P_{\alpha}[u]$. For $\beta \in u$ such that $ P_{\alpha}(\beta)=y$, $y \oplus P_{\alpha}(\beta) = 0$, a contradiction with the fact that $ y \in L_{\alpha}(u)$.
	\end{proof}
	
	Next, we study how $L_{\alpha}(u(u))$ changes from $L_{\alpha}(u)$ for certain field ordinals $u$.
	
	\begin{lemma} \label{L-general-lema-1}
		Let $\alpha < \tau $. Let $ u > \tau^\tau $ be a perfect  field ordinal, but which is not quadratically closed, such that $ u \boxedexp{2} \oplus u = \eta $, for some $\eta \in u$ and the only root of $P_{\alpha} $ in $u$ is $0$. Then:
		$$ L_{\alpha}(u(u)) = L_{\alpha}(u) \cup \{ u \times \lambda \mid \lambda \in L_{\alpha}(u) \}. $$
		In particular, $L_{\alpha}(u)$ is an initial segment for $L_{\alpha}(u)$.
	\end{lemma}
	
	\begin{proof}
		Let $\lambda \in L_{\alpha}(u(u))$. If $\lambda < u$, then $\lambda \in L_{\alpha}(u)$. If $\lambda \geq u$, there exist $ x, y <u, \ x \geq 1$ such that $\lambda = u \times x + y$. Since $\lambda $ is a group, it follows that $y = 0$, otherwise $\lambda = (u \times x) \oplus y$ and $u \times x ,\ y < \lambda$.
		
		If $x \notin L_{\alpha}(u)$, then there exist $ \beta \in u$, $n \in \omega$, $x_1,\dots , x_n \in x $ such that $x \oplus P_{\alpha}(\beta) = x_1 \oplus \dots \oplus x_n$. Let $\eta_1 = u \otimes x_1 \oplus (\tau \oplus \alpha )  \otimes \beta \boxedexp{4} \otimes (\eta \oplus \eta \boxedexp{2}) \oplus y$, $\eta_2 = u \otimes x_2, \dots, \eta_n = u \otimes x_n $. Since $\tau + \alpha, \beta, \eta , y < u$ and $x_1, \dots , x_n < x$ it follows that $ \eta_1, \dots, \eta_n < \lambda $.
		But $\lambda \oplus \eta_1 \oplus \dots \oplus \eta_n = P_{\alpha}(u \otimes \beta )$, a contradiction with $\lambda \in L_{\alpha}(u(u))$. Thus $\lambda \in L_{\alpha}(u) \cup \{ u \times \lambda \mid \lambda \in L_{\alpha}(u) \}$.
		
		Let $\lambda \in L_{\alpha}(u)$. Then $\lambda $ is a group. Suppose there exist $\beta \in u(u)$, $s < \lambda$ such that  $P_{\alpha}(\beta) = \lambda \oplus s $. There exist $x, y <u$ such that $\beta = u \otimes x \oplus y $. It follows that $u \otimes ((\tau \oplus \alpha ) \otimes x \boxedexp{4} \oplus x) \oplus ((\tau \oplus \alpha ) \otimes (x \boxedexp{4} \otimes (\eta \oplus \eta\boxedexp{2}) \oplus y \boxedexp{4}) \oplus y \oplus \lambda \oplus s ) = 0 $, so $ (\tau \oplus \alpha ) \otimes x \boxedexp{4} \oplus x =0 $. Since $x < u $, from the hypothesis it follows that $x=0$. It follows that $\beta < u$, a contradiction with $ \lambda \in L_{\alpha}(u) $.
		
		Let $a \in  L_{\alpha}(u)$. We prove that $\lambda = u \times a \in  L_{\alpha}(u(u))$. From Lemma~\ref{lema-produs-grupuri}, we have that $\lambda$ is a group. Suppose there exist  $ \lambda'  < \lambda $, $ \beta \in u(u) $ such that $\lambda \oplus \lambda' = P_{\alpha}(\beta)$. There exist $ x, y < u $, $a' < a$, $ b<u $ such that $\beta = u \otimes x \oplus y$, $\lambda' = u \otimes a' \oplus b$. Then $ a \oplus a' =( \tau \oplus \alpha) \otimes x \boxedexp{4} \oplus x $, a contradiction with $ a \in  L_{\alpha}(u) $.    
		
	\end{proof}

	We now show a description for $L_{\alpha}(\sup \{ u_\eta \mid \eta < \delta \})$ for field ordinals $(u_{\eta})_{\eta < \delta}$ that is similar to the one for $L$, proved in Theorem~\ref{claim-inch-patr}.
	
	\begin{lemma} \label{L-general-lema-2}
		Let $\theta < \tau $. Let $\delta $ be a limit ordinal. Let $ (u_\eta)_{\eta \leq \delta} $ be field ordinals such that $ u_{\delta} = \sup\{u_{\eta} \mid \eta<\delta \}$.
		Then: 
		$$
		L_{\theta}(u_{\delta}) = \bigcup_{\alpha < \delta} \bigcap_{\delta>\beta \geq \alpha} L_{\theta}(u_{\beta});
		$$
		
		$$
		L_{\theta}(u_{\delta}) = \bigcap_{\alpha < \delta} \bigcup_{\delta>\beta \geq \alpha} L_{\theta}(u_{\beta}).
		$$
		
	\end{lemma}
	
	\begin{proof}
		
		Let $\lambda \in L_{\theta}(u_{\delta})$. Then $\lambda$ is a group, $\{  \lambda + s \mid s < \lambda  \}\cap P_{\theta}[u_{\delta}] = \emptyset$. There exists $\alpha < \delta $ such that $\lambda \in u_\alpha$, so $\lambda \in L_{\theta}(u_{\alpha})$. For any $ \delta > \beta \geq \alpha$, we have that $\lambda \in u_{\beta}$, $\{  \lambda + s \mid s < \lambda  \}\cap P_{\theta}[u_{\beta}]  \subset     \{  \lambda + s \mid s < \lambda  \}\cap P_{\theta}[u_{\delta}] = \emptyset$, so $\lambda \in L_{\theta}(u_{\beta})$. Thus $L_{\theta}(u_{\delta}) \subset \bigcup_{\alpha < \delta} \bigcap_{\delta>\beta \geq \alpha} L_{\theta}(u_{\beta})$ and $ L_{\theta}(u_{\delta}) \subset  \bigcap_{\alpha < \delta} \bigcup_{\delta>\beta \geq \alpha} L_{\theta}(u_{\beta}) $.
		
		Let $\lambda \in \bigcup_{\alpha < \delta} \bigcap_{\delta>\beta \geq \alpha} L_{\theta}(u_{\beta}) $. Then there exists $\alpha < \delta $ such that $\lambda \in \bigcap_{\delta>\beta \geq \alpha} L_{\theta}(u_{\beta})$. It follows that $ \{  \lambda + s \mid s < \lambda  \}\cap P_{\theta}[u_{\beta}] = \emptyset$, for any $ \delta > \beta \geq \alpha$. Since $\bigcup_{\delta>\beta \geq \alpha} P_{\theta}[u_{\beta}] = P_{\theta}[u_{\delta}] $, it follows that $ \{  \lambda + s \mid s < \lambda  \}\cap P_{\theta}[u_{\delta}] = \emptyset$.
		Obviously $\lambda$ is a group, $\lambda \in u_{\delta}$, so $ \lambda \in L_{\theta}(u_{\delta}) $. 
		
		Let $ \lambda \in \bigcap_{\alpha < \delta} \bigcup_{\delta>\beta \geq \alpha} L_{\theta}(u_{\beta}) $. 
		Then, for any $\alpha < \delta$, there exists $ \alpha \leq \beta_{\alpha} < \delta $ such that $ \lambda \in L_{\theta}(u_{\beta_{\alpha}}) $.
		So $\lambda \in u_{\delta}$ and is a group and $ \{  \lambda + s \mid s < \lambda  \}\cap P_{\theta}[u_{\beta_{\alpha}}] = \emptyset$, for any $\alpha < \delta$. Since $\bigcup_{\alpha < \delta} P_{\theta}[u_{\beta_{\alpha}}] = P_{\theta}[u_{\delta}] $, it follows that $ \{  \lambda + s \mid s < \lambda  \}\cap P_{\theta}[u_{\delta}] = \emptyset$. So $\lambda \in L_{\theta}(u_{\delta})  $.
	\end{proof}
	
	\begin{lemma} \label{lema-A}
		Let $\theta < \tau$. Let $\delta $ be an ordinal. Consider $ (u_\eta)_{\eta \leq \delta} $ field ordinals such that, for any $\eta < \delta $, $u_{\eta} \leq u_{\eta+1}$, and $  L_{\theta}(u_{\eta})$ is an initial segment for $  L_{\theta}(u_{\eta+1})$ and for any $\alpha \leq \delta$ limit ordinal $ u_{\alpha} = \sup\{ u_{\eta} \mid \eta < \alpha  \} $. Then:
		
		\begin{itemize}
			\item for all $\alpha \leq \beta \leq \delta $, $L_{\theta}(u_{\alpha})$ is an initial segment of $L_{\theta}(u_{\beta})$;
			
			\item for all $\alpha \leq \delta  $ limit ordinal, $L_{\theta}(u_{\alpha}) = \bigcup_{\eta< \alpha} L_{\theta}(u_{\eta})  $, so   $\otype(L_{\theta}(u_\alpha)) = \sup \{ L_{\theta}(u_{\eta}) \mid \eta < \alpha  \}$.
		\end{itemize}
	\end{lemma}
	
	\begin{proof}
		Let $\alpha \leq \delta$.
		We prove the first bullet by induction on $\beta$. For $\beta=\alpha$ it is obvious. If $L_{\theta}(u_{\alpha})$ is an initial segment of $L_{\theta}(u_{\beta})$, we prove that $L_{\theta}(u_{\alpha})$ is an initial segment of $L_{\theta}(u_{\beta+1})$. From the hypothesis, we know that $L_{\theta}(u_{\beta})$ is an initial segment of $L_{\theta}(u_{\beta+1})$, so $L_{\theta}(u_{\alpha})$ is an initial segment of $L_{\theta}(u_{\beta+1})$. If $\beta>\alpha$ is a limit ordinal, we show that  $L_{\theta}(u_{\alpha})$ is an initial segment of $L_{\theta}(u_{\beta})$. Let $x \in L_{\theta}(u_{\alpha})$. Then $x \in \bigcap_{\beta>\eta\geq \alpha} L_{\theta}(u_{\eta})$, thus from Lemma~\ref{L-general-lema-2}, $x \in L_{\theta}(u_{\beta})$. 
		
		Let $x \in L_{\theta}(u_\alpha)$ and $y \in  L_{\theta}(u_\beta) $ be such that $y<x$. From Lemma~\ref{L-general-lema-2}, there exists $\alpha \leq \eta < \beta $ such that $y \in L_{\theta}(u_\eta)$. From the induction hypothesis, $y \in L_\theta(u_\alpha)$.
		
		We now prove the second bullet point. The reverse inclusion is obvious. We now prove the forward one. Let $x \in L_{\theta}(u_{\alpha})$. From Lemma~\ref{L-general-lema-2}, there exists $\eta< \alpha$, such that for all $\eta \leq \gamma < \alpha $, $x \in L_{\theta}(u_\gamma)$. In particular, for $\gamma := \eta$, $x \in L_\theta(u_\eta)$.
		
	\end{proof}

	\begin{lemma} \label{lema-B}
		Let $\theta < \tau $. Consider $ (u_\eta)_{\eta \leq \delta} $ field ordinals such that, for any $\eta \leq \delta $, $ u_{\eta} > \tau^\tau $ is a perfect field ordinal, but which is not quadratically closed, so we denote $ a_\eta := u_{\eta} \boxedexp{2} \oplus u_{\eta} \in u_{\eta} $, and the only root of $P_{\theta} $ in $u_{\eta}$ is $0$. Moreover, suppose that $u_{\eta+1} = u_{\eta}(u_{\eta})$, for any $\eta < \delta$ and $ u_{\eta} = \sup\{ u_{\eta'} \mid \eta' < \eta  \} $, for any $\eta \leq \delta$ limit ordinal. Then:
		
		\begin{itemize}
			\item the conclusions of Lemma~\ref{lema-A} hold;
			\item for any $\eta \leq \delta$,  $\otype(L_{\alpha}(u_{\eta})) = \otype(L_{\alpha}(u_0)) \times 2^\eta$.
		\end{itemize}

	\end{lemma}

	\begin{proof}
		The first bullet point follows from Lemma~\ref{L-general-lema-1} and Lemma~\ref{lema-A}.
		
		For the second bullet, from Lemma~\ref{L-general-lema-1}, we have that $\otype(L_{\alpha}(u_{\eta+1}))= \otype(L_{\alpha}(u_{\eta})) \times 2 $, for any $\eta < \delta$. From the first bullet point, we have that $\otype(L_{\alpha}(u_{\eta}))  =  \sup\{  \otype(L_{\alpha}(u_{\eta'})) \mid \eta' < \eta \}$, for any $\eta \leq \delta$ limit ordinal. The conclusion follows immediately by induction on $\eta$.
		
	\end{proof}

	\begin{lemma} \label{pol-grad3-ext-grad2}
		Let $u$ be a field ordinal. Let the sequence of field ordinals $(u_{\eta})_{\eta \leq y }$ be such that $u_0 = u$, $ u_{\eta + 1} = u_{\eta}(u_{\eta}) $ and $[u_{\eta+1}: u_{\eta} ] = 2$, for any $\eta<y$, $ u_{\eta} = \sup \{ u_{\delta} \mid \delta < \eta \} $, if $\eta$ is a limit ordinal. Let $Q \in u[X]$ be of degree 3. Then $Q$ has roots in $u_y$ if and only if it has roots in $u$.
	\end{lemma}
	
	\begin{proof}
		Suppose that $Q$ has no roots in $u$ and we prove that $Q$ has no roots in $u_y$.
		
		We prove by induction on $ \eta \leq y $ that there is no $x \in u_{\eta}$ such that $Q(x) = 0$. For $\eta = 0$ it is obvious. If there is no $x \in u_{\eta}$ such that $Q(x) = 0$, then since $Q$ has degree 3, it follows that $Q$ is irreducible over $u_{\eta}$. If there were an $x \in u_{\eta+1}$ such that $Q(x) = 0$, then $[u_{\eta}(x): u_{\eta}] = 3$, but $[u_{\eta+1}: u_{\eta} ] = 2$ and $ u_{\eta} \subset u_{\eta}(x) \subset u_{\eta+1} $, a contradiction. For $\eta$ a limit ordinal it is obvious from $ u_{\eta} = \sup \{ u_{\delta} \mid \delta < \eta \} $.
	\end{proof}
	
	Now we have all the results necessary for proving that $\{\varepsilon_\alpha \mid \alpha \leq \omega^{\omega^\omega} \}$ are the next quadratically closed field ordinals.

	\begin{theorem} \label{eps-corpuri-patratic-inchise}
		For any $ \alpha < \tau $, we have that:
		
		\begin{enumerate} [i)]
			\item The ordinal $\varepsilon_\alpha$ is the $\alpha$-th quadratically closed field ordinal after $\varepsilon_0$;
			\item For 
			any $\gamma \in \{ (\tau \oplus \beta_1)\boxedexp{m_1} \otimes \dots \otimes (\tau \oplus \beta_n)\boxedexp{m_n} \mid n >0, \ m_1, \dots, m_n \in \mathbb{Z} \text{ not all divisible by 3}, \ \tau > \beta_1, \dots, \beta_n \geq \alpha  \}$
			the polynomial $X\boxedexp{3} \oplus \gamma $ has no roots in $\varepsilon_\alpha$;
			
			\item For any $\tau > \beta \geq \alpha$, if $\alpha$ is a successor ordinal, then $ L_{\beta}(\varepsilon_{\alpha'}) $ is an initial segment for $ L_{\beta}(\varepsilon_{\alpha}) $, where $\alpha = \alpha'+1$; 
			
			\item For any $\tau > \beta \geq \alpha$, $\otype(L_{\beta}(\varepsilon_{\alpha})) = \varepsilon_{\alpha}$;
			\item $\varepsilon_\alpha \boxedexp{3} = \tau \oplus \alpha = \tau + \alpha $.
			
		\end{enumerate}
	\end{theorem}

	\begin{proof}
		We prove by induction on $\alpha$.
		
		\underline{For $\alpha = 0$:} 
		
		\begin{enumerate} [i)]
			\item It is obviously true.

			\item Let $\gamma \in \{ (\tau \oplus \beta_1)\boxedexp{m_1} \otimes \dots \otimes (\tau \oplus \beta_n)\boxedexp{m_n} \mid n >0, \ m_1, \dots, m_n \in \mathbb{Z} \text{ not all divisible by 3}, \ \tau >  \beta_1, \dots \beta_n \geq 0  \}$. From Proposition~\ref{inch-perfecta-ord-corp-alg-inchis}, Theorem~\ref{claim-inch-patr} and Lemma~\ref{pol-grad3-ext-grad2}, it is enough to prove that the polynomial $X\boxedexp{3} \oplus \gamma  $ has no roots in $\tau(\tau)$, which is obvious.
			\item There is nothing to prove.
			
			\item Let $\beta < \tau $.
			We prove that $L_{\beta}(\tau(\tau)) \neq \emptyset$.
			
			If, by contradiction, there exist $p, q \in \tau[X]$ such that $(p,q)=1$, $ q \neq 0$ monic and $P_{\beta}(p(\tau) \oslash q(\tau)) \oplus \tau \boxedexp{2} =0  $, then: $$(\tau \oplus \beta) \otimes p(\tau)\boxedexp{4} \oplus p(\tau) \otimes q(\tau)\boxedexp{3} \oplus \tau \boxedexp{2} \otimes q(\tau) \boxedexp{4} = 0.$$
			We have that $(X \oplus \beta) \otimes p \boxedexp{4} \oplus p \otimes q \boxedexp{3} \oplus X \boxedexp{2} \otimes q \boxedexp{4} =0 $. Thus $ q \boxedexp{3} \mid (X \oplus \beta) \otimes p \boxedexp{4} $, $(p,q) = 1$ from which it follows that $q \boxedexp{3} \mid (X \oplus \beta) $, so $q=1$. We have that $(X \oplus \beta) \otimes p \boxedexp{4} \oplus p \oplus X \boxedexp{2}  = 0 $. If $p$ is constant, then $(X \oplus \beta) \otimes p \boxedexp{4} \oplus p \oplus X \boxedexp{2}$ has degree 2, and if $\deg(p)>0$, then $(X \oplus \beta) \otimes p \boxedexp{4} \oplus p \oplus X \boxedexp{2}$ has degree $4 \deg(p) + 1$, a contradiction. Thus $\tau^2 \in \tau(\tau) \setminus P_{\beta}(\tau(\tau))$. From Proposition~\ref{minim-L-general} it follows that $ L_{\beta}(\tau(\tau)) \neq \emptyset$. From Remark~\ref{obs-l-general-tau} it follows that $ L_{\beta}(\tau^{\tau \times \omega}) \neq \emptyset $. Consider the sequence of fields $(u_{\eta})_{\eta \leq \varepsilon_0}$ such that $u_0 = \tau^{\tau \times \omega}$, $ u_{\eta + 1} = u_{\eta}(u_{\eta}) $, for any $\eta \leq \varepsilon_0$, $ u_{\eta} = \sup \{ u_{\delta} \mid \delta < \eta \} $, if $\eta$ is a limit ordinal. We know from Theorem~\ref{claim-inch-patr} and Remark~\ref{obs-tau-inch-patr}, that $ u_{\varepsilon_0} = \varepsilon_0 $. From Lemma~\ref{lema-B} and from Lemma~\ref{lema-epsilon}, it follows that $ \otype(L_{\beta}(\varepsilon_0)) = \otype(L_{\beta}(u_{\varepsilon_0})) = \otype{L_{\beta}(u_{0}}) \times 2^{\varepsilon_0} = \varepsilon_0$.

			\item It follows from Remark~\ref{obs-tau-radical-de-ordin-3}.

		\end{enumerate}
		
		\underline{For $\alpha^+$:}
		
		Let the sequence of fields $(v_{\eta})_{\eta \leq \varepsilon_{\alpha+1} }$, such that $v_0 = \varepsilon_\alpha(\varepsilon_\alpha)$, $ v_{\eta + 1} = v_{\eta}(v_{\eta}) $, for any $\eta<\varepsilon_{\alpha+1}$, $ v_{\eta} = \sup \{ v_{\delta} \mid \delta < \eta \} $, if $\eta \leq \varepsilon_{\alpha+1}$ is a limit ordinal. 
		
		\begin{enumerate}[i)]
			\item We prove that $L(\varepsilon_{\alpha}(\varepsilon_{\alpha}))  = \{ \varepsilon_{\alpha} \times \lambda \mid \lambda \in L_{\alpha}(\varepsilon_{\alpha})  \} $.

			Let $\lambda \in L(\varepsilon_{\alpha}(\varepsilon_{\alpha}))$. There exist $ x, y, z < \varepsilon_{\alpha} $ such that $\lambda = \varepsilon_{\alpha}^2 \times x + \varepsilon_{\alpha} \times y + z$. Since $\lambda > 0$ is a group, it follows that exactly one of $x, y,z$ is nonzero. If $z >0$, then $\lambda < \varepsilon_{\alpha} $, so $\lambda 
			\in P[\varepsilon_{\alpha}] $, a contradiction.
			
			If $y>0$ and $y \notin L_{\alpha}(\varepsilon_{\alpha}) $ it follows that there exist $y'<y$, $b<\varepsilon_{\alpha}$ such that $y \oplus y' = P_{\alpha}(b)$. Then $\lambda \oplus (\varepsilon_{\alpha} \otimes y') = P(\varepsilon_{\alpha} \boxedexp{2} \otimes b\boxedexp{2} \oplus \varepsilon_{\alpha} \otimes b )$ and $\varepsilon_{\alpha} \otimes y' < \lambda$, a contradiction with $\lambda \in L(\varepsilon_{\alpha}(\varepsilon_{\alpha}))$.
			
			If $x>0$, then $\lambda \oplus  (\varepsilon_{\alpha} \otimes (\tau \oplus \alpha) \otimes x \boxedexp{2}) = P(\varepsilon_{\alpha} \boxedexp{2} \otimes x)    $ and $ (\tau \oplus \alpha) \otimes x \boxedexp{2} < \varepsilon_{\alpha} $ , so  $\varepsilon_{\alpha} \otimes (\tau \oplus \alpha) \otimes x \boxedexp{2} < \lambda$, a contradiction with $\lambda \in L(\varepsilon_{\alpha}(\varepsilon_{\alpha}))$.
			
			Thus $L(\varepsilon_{\alpha}(\varepsilon_{\alpha}))  \subset \{ \varepsilon_{\alpha} \times \lambda \mid \lambda \in L_{\alpha}(\varepsilon_{\alpha})  \} $.
			
			Let $ y \in  L_{\alpha}(\varepsilon_{\alpha}) $. Suppose that $ \lambda = \varepsilon_{\alpha} \times y \notin L(\varepsilon_{\alpha}(\varepsilon_{\alpha})) $. Then, since $\lambda$ is a group, from Lemma~\ref{lema-produs-grupuri}, there exist $s<\lambda$, $\beta<\varepsilon_{\alpha}^3$ such that $P(\beta) =\lambda \oplus s $. There exist $ y'<y $, $ v<\varepsilon_{\alpha} $, $a,b,c < \varepsilon_{\alpha}$ such that $s= \varepsilon_{\alpha} \times y' + v$ and $ \beta = \varepsilon_{\alpha}^2 \times a + \varepsilon_{\alpha} \times b + c $. Then, from the relation $ P(\beta) =\lambda \oplus s  $ it follows that $ 0 = b\boxedexp{2} \oplus a$ and $ y \oplus  y' = (\tau \oplus \alpha) \otimes a\boxedexp{2} \oplus b $, so $P_{\alpha}(b) = y \oplus y'$, a contradiction with $y \in L_{\alpha}(\varepsilon_{\alpha})$. 
			
			Thus $L(\varepsilon_{\alpha}(\varepsilon_{\alpha}))  \supset \{ \varepsilon_{\alpha} \times \lambda \mid \lambda \in L_{\alpha}(\varepsilon_{\alpha})  \} $.

			By induction, $\otype(L_{\alpha}(\varepsilon_{\alpha})) = \varepsilon_{\alpha} $, so $\otype(L(\varepsilon_{\alpha}(\varepsilon_{\alpha}))) = \varepsilon_{\alpha}$. From Theorem~\ref{claim-inch-patr}, since $ \varepsilon_{\alpha}(\varepsilon_{\alpha}) = \varepsilon_\alpha^3 $ is a perfect field, it follows that $(\varepsilon_{\alpha}^3)^{\varepsilon_{\alpha+1}}$ is the quadratic closure of $\varepsilon_{\alpha}(\varepsilon_{\alpha})$. But, from Lemma~\ref{lema-epsilon}, $(\varepsilon_{\alpha}^3)^{\varepsilon_{\alpha+1}} = \varepsilon_{\alpha}^{\varepsilon_{\alpha+1}} = 2^{\varepsilon_{\alpha} \times \varepsilon_{\alpha+1}} = 2^{\varepsilon_{\alpha+1}} = \varepsilon_{\alpha+1}$. Thus $\varepsilon_{\alpha+1}$ is the $(
		\alpha+1)$-th quadratically closed field ordinal after $\varepsilon_0$. 
			
			In particular, $v_{\varepsilon_{\alpha+1}}= \varepsilon_{\alpha+1}$.
			
			\item Let $\gamma \in \{ (\tau \oplus \beta_1)\boxedexp{m_1} \otimes \dots \otimes (\tau \oplus \beta_n)\boxedexp{m_n} \mid n >0, \ m_1, \dots, m_n \in \mathbb{Z} \text{ not all divisible by 3}, \ \tau > \beta_1, \dots, \beta_n > \alpha  \}$.
			
			From Lemma~\ref{pol-grad3-ext-grad2} applied to $(v_{\eta})_{\eta \leq \varepsilon_{\alpha+1} }$, it is enough to prove that $X \boxedexp{3} \oplus \gamma $ has no roots in $ \varepsilon_\alpha(\varepsilon_\alpha) $. Suppose towards a contradiction that there exist $ a,b,c < \varepsilon_\alpha $, such that $ ( \varepsilon_{\alpha} \boxedexp{2} \otimes a \oplus \varepsilon_\alpha \otimes b \oplus c) \boxedexp{3} = \gamma $. Then:
			$$  b\boxedexp{2} \otimes c \oplus (\tau \oplus \alpha)\otimes b \otimes a \boxedexp{2} \oplus a \otimes c\boxedexp{2} =0;   $$
			$$   (\tau \oplus \alpha) \otimes b \boxedexp{2} \otimes a \oplus (\tau \oplus \alpha) \otimes a \boxedexp{2} \otimes c \oplus b \otimes c \boxedexp{2} =0;    $$
			$$  (\tau \oplus \alpha) \otimes b \boxedexp{3} \oplus (\tau \oplus \alpha ) \boxedexp{2} \otimes a \boxedexp{3} \oplus c \boxedexp{3} = \gamma. $$
			
			Multiplying the first relation by $b$, the second by $a$ and adding them, it follows that:
			$$ c \otimes (b \boxedexp{3} \oplus (\tau \oplus \alpha) \otimes a \boxedexp{3}) =0. $$ 
			If $c \neq 0$,
			since $ X \boxedexp{3} \oplus (\tau \oplus \alpha) $ has no roots in $\varepsilon_\alpha$, by induction, then, since $b \boxedexp{3} \oplus (\tau \oplus \alpha) \otimes a \boxedexp{3} = 0$, we have $a=b=0$, 
			so $c \boxedexp{3 } = \gamma$, contradiction with ii) from the induction hypothesis for $\alpha$.
			
			Thus $c =0 $, from which it follows from the first relation that $a\otimes b =0$.
			If $a=0$, then $b \boxedexp{3} = \gamma \oslash (\tau \oplus \alpha)$, contradiction with ii) from the induction hypothesis for $\alpha$.
			If $b=0$, then $a \boxedexp{3} = \gamma \oslash (\tau \oplus \alpha)\boxedexp{2}$, contradiction with ii) from the induction hypothesis for $\alpha$. 
			
			In particular, for all $ \alpha < \beta < \tau $, $P_{\beta}$ has no nonzero roots in $\varepsilon_{\alpha+1}$. 
			
			\item Let $\tau > \beta > \alpha $. 
			From Lemma~\ref{lema-B} applied to $(v_{\eta})_{\eta \leq \varepsilon_{\alpha+1} }$, it is enough to prove that $L_{\beta}(\varepsilon_\alpha)$ is an initial segment for $L_{\beta}(\varepsilon_\alpha(\varepsilon_\alpha))$. 
			
			Let $r \in L_{\beta}(\varepsilon_\alpha) $. Then $r$ is a group. Suppose there exist $a,b,c < \varepsilon_{\alpha} $, $r'<r$ such that $ P_{\beta} (\varepsilon_\alpha \boxedexp{2} \otimes a \oplus \varepsilon_\alpha \otimes b \oplus c) = r \oplus r' $. Then, it follows that $P_{\beta}(c)=r \oplus r'$, contradiction with $r \in L_{\beta}(\varepsilon_\alpha) $. Thus $L_{\beta}(\varepsilon_\alpha) \subset L_{\beta}(\varepsilon_\alpha(\varepsilon_\alpha))$.
			
			Let $r \in L_{\beta}(\varepsilon_\alpha(\varepsilon_\alpha)) \setminus L_{\beta}(\varepsilon_\alpha) $. Obviously $r \geq \varepsilon_{\alpha}$ and since $L_{\beta}(\varepsilon_\alpha)  \subset \varepsilon_\alpha$, it follows that $L_{\beta}(\varepsilon_\alpha)$ is an initial segment for $L_{\beta}(\varepsilon_\alpha(\varepsilon_\alpha))$.

			\item Let $\tau > \beta > \alpha $. 
			From iii), we have that $L_{\beta}(\varepsilon_\alpha(\varepsilon_\alpha)) \neq \emptyset$.

		 From Lemma~\ref{lema-B} applied to $(v_{\eta})_{\eta \leq \varepsilon_{\alpha+1} }$, and from Lemma~\ref{lema-epsilon} it follows that $\otype(L_{\beta}(v_{\varepsilon_{\alpha+1}})) = \otype{L_{\beta}(v_{0}}) \times 2^{\varepsilon_{\alpha+1}} = \varepsilon_{\alpha+1}$.
			
			\item Obvious from i), ii) and Theorem~\ref{set} g).

		\end{enumerate}
		
		\underline{For $\alpha$ a limit ordinal:}
		
		\begin{enumerate}[i)]
			\item From i) from the induction hypothesis, it follows that $ \varepsilon_\alpha $ is the $\gamma$-th quadratically closed field ordinal after $ \varepsilon_0 $, where $ \gamma \geq \alpha $. If $ \chi $ is the $\alpha$-th quadratically closed field ordinal after $ \varepsilon_0 $ and $ \chi < \varepsilon_\alpha $, then there exists $\beta<\alpha$ such that $ \chi < \varepsilon_{\beta} $, from which it would follow that $ \alpha < \beta $, a contradiction. Thus $\gamma = \alpha$.

			\item Let $\gamma \in \{ (\tau \oplus \beta_1)\boxedexp{m_1} \otimes \dots \otimes (\tau \oplus \beta_n)\boxedexp{m_n} \mid n >0, \ m_1, \dots, m_n \in \mathbb{Z} \text{ not all divisible by 3}, \ \tau > \beta_1, \dots, \beta_n \geq \alpha  \}$. From ii) from the induction hypothesis, we know that
			the polynomial $X\boxedexp{3} \oplus \gamma $ has no roots in $\varepsilon_\beta$, for any $\beta < \alpha$. It follows that the polynomial $X\boxedexp{3} \oplus \gamma $ has no roots in $\varepsilon_\alpha$.
			
			\item There is nothing to prove.

			\item Let $ \tau > \beta \geq \alpha $. From iv) from the induction hypothesis we know that for any $\eta < \alpha$, $\otype(L_{\beta}(\varepsilon_{\eta})) = \varepsilon_{\eta}$. From iii) from the induction step, and Lemma~\ref{lema-A}, it follows that $ \otype(L_{\beta}(\varepsilon_\alpha)) = \sup \{ \otype(L_{\beta}(\varepsilon_{\eta})) \mid \eta < \alpha \} $. Thus $\otype(L_{\beta}(\varepsilon_\alpha)) = \varepsilon_\alpha$.
			
			\item Obvious from i), ii) and Theorem~\ref{set} g).
			
		\end{enumerate}		
	\end{proof}

	From Theorem~\ref{eps-corpuri-patratic-inchise} it follows that $ \varepsilon_\tau $ is a quadratically closed field.
	
	In the following, we will find the smallest irreducible polynomial over $\varepsilon_\tau$.

	\begin{lemma} \label{lema-eps-tau-rel-1}
		For any $\alpha < \tau ^{\tau \times \omega}$, there exists $ \beta \in \varepsilon_\tau $ such that $\beta \boxedexp{3} = \alpha $.
	\end{lemma}
	
	\begin{proof}
		For $\alpha < \tau$ it is obvious. For $\alpha = \tau + x$, $x < \tau$ it is true from Theorem~\ref{eps-corpuri-patratic-inchise}. Since $\tau ^\tau = \tau(\tau)$, it follows that the lemma is true for any $\alpha \in \tau ^ \tau$. 
		
		Since $\tau ^ \tau \hookrightarrow \tau^{\tau \times \omega}$ is purely inseparable by Proposition~\ref{corolar-pur-insep} and Proposition~\ref{inch-perfecta-ord-corp-alg-inchis}, then from Proposition~\ref{pur-insep} it follows that for any $ \alpha < \tau^{\tau \times \omega} $ there exists $n \in \omega$ such that $\alpha \boxedexp{2^n} \in \tau^\tau$. Thus there exists $\beta \in \varepsilon_\tau $ such that $ \beta \boxedexp{3} =  \alpha \boxedexp{2^n}$. Since $\varepsilon_\tau$ is quadratically closed, there exists $ \gamma \in \varepsilon_\tau $ such that $\gamma \boxedexp{2^n} = \beta$, so $\gamma \boxedexp{3} = \alpha$. 
	\end{proof}
	
	\begin{lemma} \label{lema-eps-tau-rel-2}
		There is no $a \in \varepsilon_\tau$ such that $a \boxedexp{3} = \tau ^{\tau \times \omega}$.
	\end{lemma}
	
	\begin{proof}
		We prove by induction on $\alpha \leq \tau $ that for no $\eta \in \tau(\tau) \setminus \{0\}  $ there exists $ a \in \varepsilon_\alpha $ such that $ a \boxedexp{3} = \tau^{\tau \times \omega } \otimes \eta $.
		
		For the base case $\alpha = 0$, 
		we prove that for no $\eta \in \tau(\tau) \setminus \{0\}  $ there exists $ a \in \tau^{\tau \times \omega }(\tau^{\tau \times \omega }) $ such that $ a \boxedexp{3} = \tau^{\tau \times \omega } \otimes \eta $. Let $\eta \in \tau(\tau) \setminus \{0\}  $. Assume there exists $ a \in \tau^{\tau \times \omega }(\tau^{\tau \times \omega }) $ such that $ a \boxedexp{3} = \tau^{\tau \times \omega } \otimes \eta $.
		Then $ a = \tau^{\tau \times \omega} \times x +y $, where $x, y < \tau^{\tau \times \omega}$. From Proposition~\ref{rel-inch-perf} we know that $ (\tau^{\tau \times \omega}) \boxedexp{2} \oplus \tau^{\tau \times \omega} \oplus \tau =0 $. Thus $a \boxedexp{3} = (\tau^{\tau \times \omega}) \boxedexp{3} \otimes x \boxedexp{3} \oplus (\tau^{\tau \times \omega}) \boxedexp{2} \otimes x \boxedexp{2} \otimes y \oplus \tau^{\tau \times \omega} \otimes x \otimes y \boxedexp{2} \oplus y \boxedexp{3} = \tau^{\tau \times \omega} \otimes \eta $, from which it follows that $ \tau^{\tau \times \omega} \otimes ( (\tau \oplus 1) \otimes x \boxedexp{3} \oplus x \boxedexp{2} \otimes y \oplus x \otimes y \boxedexp{2} ) \oplus (\tau \otimes x \boxedexp{3} \oplus \tau \otimes x \boxedexp{2} \otimes y \oplus y \boxedexp{3}) = \tau^{\tau \times \omega} \otimes \eta $.
		
		So $\tau \otimes x \boxedexp{3} \oplus \tau \otimes x \boxedexp{2} \otimes y \oplus y \boxedexp{3} = 0$. If $y = 0$, then $x=0$, so $a=0$, a contradiction. If $y >0$, then $\tau \otimes c \boxedexp{3} \oplus \tau \otimes c \boxedexp{2} \oplus 1 =0$, where $c = x \oslash y \in \tau^{\tau \times \omega}$.
		
		Assume there exist $ p, q \in \tau[X] $, $q$ monic, $(p,q)=1$ such that $\tau \otimes (p(\tau)\oslash q(\tau)) \boxedexp{3} \oplus \tau \otimes (p(\tau)\oslash q(\tau)) \boxedexp{2} \oplus 1 =0 $. Then $ X \otimes p \boxedexp{3} \oplus X \otimes p \boxedexp{2} \otimes q \oplus q \boxedexp{3} =0 $. Thus $X \mid q$ and $ q \mid (X \otimes p \boxedexp{3}) $, therefore $ q \mid X $. Thus $q = X$, from which it follows that $ p \boxedexp{3} \oplus X \otimes p \boxedexp{2} \oplus X \boxedexp{2} =0 $. Then $ X \mid p $, a contradiction. Thus, the polynomial $\tau \otimes T \boxedexp{3} \oplus \tau \otimes T \boxedexp{2} \oplus 1 $ has no roots in $\tau(\tau)$. 
		
		From Lemma~\ref{pol-grad3-ext-grad2} it follows that the polynomial $\tau \otimes T \boxedexp{3} \oplus \tau \otimes T \boxedexp{2} \oplus 1 $ has no roots in $ \tau^{\tau \times \omega} $, a contradiction with $c \in \tau^{\tau \times \omega}$. 
		
		Thus the polynomial $T \boxedexp{3} \oplus \tau^{\tau \times \omega} \otimes \eta$ has no roots in $ \tau^{\tau \times \omega}(\tau^{\tau \times \omega}) $. From Lemma~\ref{pol-grad3-ext-grad2} it follows that the polynomial $T \boxedexp{3} \oplus \tau^{\tau \times \omega} \otimes \eta$ has no roots in $\varepsilon_0$.
		
		We assume the induction hypothesis to be true for any $\beta \leq \alpha$ and we prove it for $\alpha^+$.
		Let $\eta \in \tau(\tau) \setminus \{0\}  $. Assume there exists $ x \in \varepsilon_{\alpha}(\varepsilon_\alpha) $ such that $ x \boxedexp{3} = \tau^{\tau \times \omega} \otimes \eta $. Then $x = \varepsilon_\alpha^2 \times a + \varepsilon_\alpha \times b + c$, where $a,b,c < \varepsilon_\alpha$. It follows that:
		$$  b\boxedexp{2} \otimes c \oplus (\tau \oplus \alpha)\otimes b \otimes a \boxedexp{2} \oplus a \otimes c\boxedexp{2} =0;   $$
		$$   (\tau \oplus \alpha) \otimes b \boxedexp{2} \otimes a \oplus (\tau \oplus \alpha) \otimes a \boxedexp{2} \otimes c \oplus b \otimes c \boxedexp{2} =0;    $$
		$$  (\tau \oplus \alpha) \otimes b \boxedexp{3} \oplus (\tau \oplus \alpha ) \boxedexp{2} \otimes a \boxedexp{3} \oplus c \boxedexp{3} = \tau^{\tau \times \omega} \otimes \eta. $$
		Multiplying the first relation by $b$, the second by $a$ and adding them, it follows that:
		$$ c \otimes (b \boxedexp{3} \oplus (\tau \oplus \alpha) \otimes a \boxedexp{3}) =0. $$

		If $c \neq 0$,
		since $ X \boxedexp{3} \oplus (\tau \oplus \alpha) $ has no roots in $\varepsilon_\alpha$, then, since $b \boxedexp{3} \oplus (\tau \oplus \alpha) \otimes a \boxedexp{3} = 0$, we have $a=b=0$, 
		so $c \boxedexp{3 } = \tau^{\tau \times \omega} \otimes \eta$, contradiction with ii) from the induction hypothesis for $\alpha$.
		
		Thus $c =0 $, from which it follows from the first relation that $a\otimes b =0$.

		If $a=0$, then $ b \boxedexp{3} = \tau^{\tau \times \omega} \otimes \eta \oslash (\tau \oplus \alpha) $, a contradiction with the induction hypothesis. If $b=0$, then $ a \boxedexp{3} = \tau^{\tau \times \omega} \otimes \eta \oslash (\tau \oplus \alpha) \boxedexp{2} $, a contradiction with the induction hypothesis. 
		
		Thus, the polynomial  $T \boxedexp{3} \oplus \tau^{\tau \times \omega} \otimes \eta$ has no roots in $\varepsilon_{\alpha}(\varepsilon_\alpha)$. From Lemma~\ref{pol-grad3-ext-grad2} and Theorem~\ref{eps-corpuri-patratic-inchise} it follows that the polynomial $T \boxedexp{3} \oplus \tau^{\tau \times \omega} \otimes \eta$ has no roots in $\varepsilon_{\alpha+1}$.
		
		The case where $\alpha$ is a limit ordinal is obvious.
	\end{proof}
	
	\begin{proposition} \label{eps_tau^3}
		We have that $\varepsilon_\tau \boxedexp{3} = \tau ^{\tau \times \omega} $.
	\end{proposition}
	
	\begin{proof}
		It follows from the fact that $\varepsilon_\tau$ is a quadratically closed field and from Lemma~\ref{lema-eps-tau-rel-1}, Lemma~\ref{lema-eps-tau-rel-2}, Theorem~\ref{set} g).
	\end{proof}

	By Lemma~\ref{reuniune-corp-n-inchis}, we know that in the search of the next transcendental ordinal we should find the next $3$-closed field ordinal. Proposition~\ref{eps_tau^3} implies that this ordinal is strictly greater than $\varepsilon_{\omega^{\omega^{\omega}}}$.

	\section{Acknowledgements}
	
	This paper is based on my bachelor thesis, written under the supervision of Andrei Sipoș, whom I would like to thank  for introducing me to Nim arithmetic and for providing the insight for Proposition~\ref{omega-1-transc}.

\end{document}